\documentclass{article}
\usepackage{graphicx} 

\usepackage{amsthm}
\usepackage{amsmath,amsfonts,amssymb}
\usepackage{a4wide}
\usepackage[usenames,dvipsnames]{color}
\usepackage[all,cmtip]{xy}
\usepackage[verbose,colorlinks=true,linktocpage=true,linkcolor=blue,citecolor=blue]{hyperref}
\usepackage{footnote}
\usepackage{tikz}
\usepackage{extarrows}
\usepackage[title]{appendix}
\usepackage{hyperref}
\usepackage{marvosym}
\usepackage{indentfirst} 
\theoremstyle{definition}
\newtheorem{definition}{Definition}[section]

\theoremstyle{plain}
\newtheorem{theorem}[definition]{Theorem}
\newtheorem{proposition}[definition]{Proposition}
\newtheorem{lemma}[definition]{Lemma}
\newtheorem{corollary}[definition]{Corollary}

\theoremstyle{remark}
\newtheorem{remark}[definition]{Remark}

\numberwithin{equation}{section}

\allowdisplaybreaks[1]

\title{Quantum Berezinian for the Twisted Super Yangian}
\author{Hongda Lin ${}^{1}$ Yongjie Wang ${}^{2,}$\thanks{Corresponding Author. Email:\ wyjie@mail.ustc.edu.cn}~~ and Honglian Zhang ${}^{3,4}$}
\date{}

\begin{document}

\maketitle
\begin{center}
\footnotesize
\begin{itemize}
\item[1] Shenzhen International Center for Mathematics, Southern University of Science and Technology, China.
\item[2] School of Mathematics, Hefei University of Technology, Hefei, Anhui, 230009, China.
\item[3] Department of Mathematics, Shanghai University, Shanghai 200444, China.
\item[4] Newtouch Center for Mathematics at Shanghai University, Shanghai 200444, China.
\end{itemize}
\end{center}

\begin{abstract}
Motivated by an open problem proposed in Molev's book \cite[Section 2.16, Example 16]{Mo07}, we investigate the quantum Berezinian $\mathfrak{B}^{tw}(u)$ associated with the twisted super Yangian, which is a coideal sub-superalgebra of the super Yangian of the general linear Lie superalgebra. We provide an explicit formulation of $\mathfrak{B}^{tw}(u)$, and we also construct the center of the twisted super Yangian. This construction enables us to define the special twisted super Yangian, which is isomorphic to the quotient of the twisted super Yangian by its center. Moreover, we demonstrate the quantum Sylvester theorem for both the generator matrix and the quantum Berezinian.
\end{abstract}

\noindent\textit{MSC(2010):} 17B37, 17D10, 20G42.
\bigskip

\noindent\textit{Keywords:} Quantum Berezinian; Twisted Super Yangian; Quantum Sylvester Theorem; Center.

\tableofcontents

\section{Introduction}
The term ``Yangian", first introduced by Drinfeld \cite{Dr85,Dr87}, is used to describe a significant class of quantum groups that are closed connected to the rational solutions of the quantum Yang-Baxter equation. It exhibits a profound integration into mathematical physics and quantum integrable systems, emerging as a prevalent topic in contemporary research.
In the context of a finite-dimensional classical simple Lie algebra $\mathfrak{a}$, the Yangian $\mathrm{Y}(\mathfrak{a})$ associated with $\mathfrak{a}$ can be regarded as a deformation of the universal enveloping algebra $\mathrm{U}(\mathfrak{a}[x])$ of the polynomial current Lie algebra $\mathfrak{a}[x]$.

Yangian admits a presentation whose defining relations are expressed in a specific matrix form, originally introduced by Faddeev-Reshetikhin-Takhtajan (FRT) \cite{FRT88}. In the FRT formalism, the Yangian is realized through a generator matrix that satisfies a set of ternary relations, known as the R-matrix presentation. These relations endow the Yangian with a natural bialgebra structure, which can be further equipped with an antipode to form a Hopf algebra. In their work \cite{BK05}, Brundan and Kleshchev developed the parabolic presentation for the Yangian of the general linear Lie algebra $\mathfrak{gl}_N$, utilizing Gauss decomposition and quasi-determinants \cite{GGRW05} of the FRT generators. This presentation has a close connection with the Drinfeld presentation, as described in \cite{Dr87}.

Twisted Yangian was introduced by Sklyanin \cite{Sk88} to investigate the boundary conditions of quantum integrable systems, and was named by Olshanskii. In \cite{Ol92}, the author defined the twisted Yangians associated to the symmetric pairs:
\begin{gather*}
    \textrm{AI}:\ (\mathfrak{gl}_N,\mathfrak{o}_N),\qquad \textrm{AII}:\ (\mathfrak{gl}_{2n},\mathfrak{sp}_{2n}),
\end{gather*}
and used them to construct the representations of infinite-dimensional classical Lie algebras. Let $\theta$ be an involutive automorphism over $\mathfrak{gl}_N$. Each of these algebras specializes to the universal enveloping algebras of a twisted polynomial current Lie algebra $\mathfrak{gl}_N[x]^{\theta}$.

There are two methods for realizing a class of twisted Yangians \cite{MNO96}: one approach defines them as coideal subalgebras of the Yangian in terms of its comultiplication, while the other is based on a presentation generated by a series of current elements, along with the quaternary relations and the symmetry relations. This framework also applies to the twisted Yangians of \textrm{AIII} type \cite{MR02} as well as those of \textrm{BCD} types \cite{GR16}.

The representation theory of Yangians and twisted Yangians has been extensively studied for over thirty years. In \cite{Mo98}, Molev classified the irreducible highest weight representations of twisted Yangians associated with classical Lie algebras $\mathfrak{o}_N$ and $\mathfrak{sp}_N$. He further applied the quantum Sylvester theorem to study the skew representation of twisted Yangian $\mathrm{Y}^{tw}\left(\mathfrak{sp}_{2n}\right)$ in \cite{Mo06}, which encompasses the classical results of evaluation representations of $\mathrm{Y}^{tw}\left(\mathfrak{sp}_{2n}\right)$. More recently, a series of works \cite{LWZ23,LWZ24,LZ24} have constructed Drinfeld presentations of twisted Yangians for arbitrary quasi-split types, thereby enabling the study of character theories of twisted Yangians, analogous to that of Yangians \cite{Kn95}.

In the early 1990s, Nazarov \cite{Na91} introduced a super version of the Yangian associated with the general linear Lie superalgebra $\mathfrak{gl}_{M|N}$, which is commonly known as the super Yangian, denoted by $\mathrm{Y}\left(\mathfrak{gl}_{M|N}\right)$. Since then, the super Yangian $\mathrm{Y}\left(\mathfrak{gl}_{M|N}\right)$ has been extensively studied by numerous researchers. Zhang \cite{Zrb96} developed  finite-dimensional irreducible representation of $\mathrm{Y}\left(\mathfrak{gl}_{M|N}\right)$ connected to the standard Dynkin diagram. Additionally, the super Yangian $\mathrm{Y}\left(\mathfrak{gl}_{M|N}^{\mathfrak{s}}\right)$ at arbitrary parity sequences $\mathfrak{s}$ has been linked by reflections \cite{Ts20}, with analogous results for non-conjugacy Dynkin diagrams established by \cite{Lk22,Mo22}. In a separate work, Gow \cite{Go05,Go07} presented a Gauss decomposition for $\mathrm{Y}\left(\mathfrak{gl}_{M|N}\right)$, provided the structure of the center $\mathcal{Z}\mathrm{Y}\left(\mathfrak{gl}_{M|N}\right)$ and explored its connection with the Drinfeld presentation of the special super Yangian $\mathrm{Y}\left(\textrm{A}(M-1,N-1)\right)$ as presented by Stukopin \cite{St94}. Peng \cite{Pe11,Pe16} derived the parabolic presentations of the super Yangian and extended them to $\mathfrak{gl}_{M|N}$ associated with arbitrary Borel subalgebras.

As a coideal of the super Yangian $\mathrm{Y}\left(\mathfrak{gl}_{M|N}\right)$, the twisted super Yangians $\mathrm{Y}^{\pm}(M|N)$ were defined in terms of a specific transposition by \cite{AACFR03,BR03}. These superalgebras are deformations of the universal enveloping superalgebras $\mathrm{U}\left(\mathfrak{osp}_{M|N}[x]\right)$. Briot-Ragoucy \cite{BR03} demonstrated that the product $MN$ cannot be odd for the superalgebras $\mathrm{Y}^{\pm}(M|N)$ and the two classes of twisted super Yangians are isomorphic upon exchanging $M$ and $N$.

It is well known that there is a close connection between the structures and representations of Yangians and quantum loop algebras. The latter is a significant class of Hopf algebras as quantum deformations of loop algebras. Drinfeld asserted \cite{Dr86} that, in a certain sense, Yangians are a limited form of quantum loop algebras. Until around 2012, comprehensive proofs of this statement were established by Gautam-Toledano Laredo \cite{GT13} and Guay-Ma \cite{GM12}, respectively. Subsequent to these proofs, analogous claims were also formulated for twisted cases \cite{CG15} and for super cases \cite{LhdWZ24}. Furthermore, in another significant paper by Gautam-Toledano Laredo \cite{GT16}, the authors also established the equivalence between the finite-dimensional representations of Yangians and quantum loop algebras for symmetrizable Kac-Moody algebras.

Inspired by numerous pioneering researches \cite{Go05,Go07,JLZ24,JZ24,Mo95,Mo07,MNO96,MRS03,Na91,Na20,Ol92}, the main objective of this paper is to investigate the central structure of the twisted super Yangian, as well as its super analogue of the Sklyanin determinant. Our comprehensive study not only delves into the fundamental definitions and properties of the twisted super Yangian, but also provides a precise explanation of the center and the quantum Berezinian.
Our endeavor potentially offers several future directions for the twisted super Yangian, involving Gauss decomposition, Drinfeld presentation,  classification of irreducible representations, Gelfand-Tsetlin bases, and Harish-Chandra isomorphism, among others.

The paper is organized as follows. In Section \ref{se:Superyangian}, we recall some fundamental concepts related to the super Yangian $\mathrm{Y}\left(\mathfrak{gl}_{M|N}\right)$, its center $\mathcal{Z}\mathrm{Y}\left(\mathfrak{gl}_{M|N}\right)$, and the special super Yangian $\mathrm{Y}\left(\mathfrak{sl}_{M|N}\right)$. In Section \ref{se:twistedsuperyangian}, we define a class of twisted super Yangians, which are associated with a nonsingular block matrix featuring a symmetric $(M\times M)$-submatrix in the upper-left block and a skew-symmetric $(N\times N)$-submatrix in the lower-right block. Additionally, we introduce a transposition of order 2, which yields an equivalent realization of the twisted super Yangian and generalizes the definition presented in \cite{BR03}. In Section \ref{se:Invo-graded}, we investigate the graded algebra for the twisted super Yangian $\mathrm{Y}^{tw}\left(\mathfrak{osp}_{M|N}\right)$ associated with the second filtration of the super Yangian. Based on this, we find a formal power series
$$\mathfrak{z}(u)\in 1+u^{-1}\mathrm{Y}^{tw}\left(\mathfrak{osp}_{M|N}\right)\left[\left[u^{-1}\right]\right]$$
such that the coefficients of $\mathfrak{z}(u)$ generate the center of $\mathrm{Y}^{tw}\left(\mathfrak{osp}_{M|N}\right)$. Furthermore, we introduce the super version of the special twisted Yangian, which is defined as a quotient of $\mathrm{Y}^{tw}\left(\mathfrak{osp}_{M|N}\right)$ ($M\neq N$) by its center. In Section \ref{se:quantumbereianian}, we examine the quantum Berezinian for the twisted super Yangian $\mathrm{Y}^{tw}\left(\mathfrak{osp}_{M|N}\right)$ and establish the corresponding quantum Liouville formula. In the final section \ref{se:quantumsylvester}, we demonstrate that a twisted super Yangian can be regarded as a quotient of an extended (associative) superalgebra by the ideal generated by central elements. This extended superalgebra is referred to as the extended twisted super Yangian. Finally, we formulate the quantum Sylvester theorem for the quantum Berezinian of the extended twisted super Yangian.

\vspace{1em}
{\bf Notations and terminologies}
\vspace{0.2em}

Throughout this paper, we will use the standard notations $\mathbb{C}$, $\mathbb{Z}$, $\mathbb{Z}_+$ and $\mathbb{N}$ to represent the sets of complex numbers, integers, non-negative integers, and positive integers, respectively. Unless otherwise stated, all matrices, (super)spaces, associative (super)algebras, and Lie (super)algebras are assumed to be defined over $\mathbb{C}$. We denote $\mathcal{Z}(\mathcal{A})$ the center of the associative (super)algebra $\mathcal{A}$. The Kronecker delta $\delta_{ij}$ is equal to $1$ if $i=j$ and $0$ otherwise.

We write $\mathbb{Z}_2=\left\{\bar{0},\bar{1}\right\}$. For a homogeneous element $x$ of an associative or Lie superalgebra, we use $|x|$ to denote the degree of $x$ with respect to the $\mathbb{Z}_2$-grading. Throughout the paper, when we write $|x|$ for an element $x$,  we always assume that $x$ is a homogeneous element and automatically extend the relevant formulas by linearity (whenever applicable). All modules of Lie superalgebras and quantum superalgebras are assumed to be $\mathbb{Z}_2$-graded. The tensor product of two superalgebras $A$ and $B$ carries a superalgebra structure defined by
$$(a_1 \otimes b_1)\cdot(a_2\otimes b_2) =(-1)^{|a_2||b_1|}a_1a_2\otimes b_1b_2,$$
for homogeneous elements $a_i\in A,$ $b_i\in B$ with $i=1,2$.

Given $M,N\in\mathbb{Z}_+$ such that $M+N\geqslant 2$, we set the index set $I=\{1,\ldots,M+N\}$, the positive index set $I_+=\{1,\ldots,M\}$ and the negative index set $I_-=I\setminus I_+$.
Let $V=\mathbb{C}^{M|N}$ be a $(M+N)$-dimensional superspace with a homogeneous basis $e_1,e_2,\cdots,e_{M+N}$. We briefly denote the degree of $e_k$ with $|k|$ for all $k\in I$. Let $\operatorname{End}V$ be the endomorphism algebra of the superspace $V$ and denote by $E_{ij}$ the endomorphism defined by $E_{ij}e_k=\delta_{jk}e_i$ for all $k\in I$.

We define $X=(X_{ij})_{i,j\in I}$ as a $(M+N)\times (M+N)$ matrix with entries in an associative superalgebra $\mathcal{A}$, if $|X_{ij}|=|i|+|j|$ for all $i,j$.
This matrix $X$ can be regarded as the element
\begin{align*}\label{tensornotation}
X=\sum\limits_{i,j\in I}E_{ij}\otimes X_{ij}\in\mathrm{End}V\otimes\mathcal{A}.
\end{align*}
Next, we define the graded transposition $t$ subject to an arbitrary matrix $X=(X_{ij})_{i,j\in I}$:
\begin{gather*}
X^{t}=\sum_{i,j\in I}E_{ij}\otimes X_{ij}^{t}\quad\text{with}\quad X_{ij}^{t}=(-1)^{|i||j|+|j|}X_{ji}.
\end{gather*}

We consider tensor product superalgebras of the form
$\mathrm{End}V^{\otimes m}\otimes \mathcal{A}$. For any $1\leqslant a\leqslant m$, let $X_a$ denote the element associated with the $a$-th copy of $\mathrm{End}V$ so that
$$X_a=\sum\limits_{i,j\in I} 1^{\otimes (a-1)}\otimes E_{ij}\otimes 1^{\otimes (m-a)}\otimes X_{ij}\in\mathrm{End}V^{\otimes m}\otimes \mathcal{A}.$$
Let $t_a$ be the partial transposition obtained from the action of $t$ on the $a$-th position of tensor products.
The supertrace map is $\mathrm{Str}:\ \mathrm{End}V\longrightarrow\mathbb{C}$ defined by $E_{ij}\mapsto (-1)^{|i|}\delta_{ij}.$
Furthermore, for any $a\in\{1,\ldots,m\}$ the partial supertrace $\mathrm{Str}_{a}$ is defined as follows:
$$\mathrm{Str}_{a}:\ \mathrm{End}V^{\otimes m}\longrightarrow\mathrm{End}V^{\otimes (m-1)},$$
which applies the supertrace to the $a$-th copy of $\mathrm{End}V$ and acts the identity map on all the remaining copies. The full supertrace $\mathrm{Str}_{1,\cdots,m}$ is the composition $\mathrm{Str}_{1}\circ\cdots\circ\mathrm{Str}_{m}$.

We also use the following super-permutation operator:
\begin{gather*}
P=\sum_{i,j\in I}(-1)^{|j|}E_{ij}\otimes E_{ji}\in \operatorname{End}V\otimes \operatorname{End}V,
\end{gather*}
which acts on $V\otimes V$ via $P(e_i\otimes e_j)=(-1)^{|i||j|}e_j\otimes e_i$ for $i,j\in I$. For a $R=\sum r_{[1]i}\otimes r_{[2]i}\in \operatorname{End}V^{\otimes 2}$ and an integer $m\geqslant 2$, denote for $1\leqslant a<b\leqslant m$,
\begin{gather*}
    R_{ab}=\sum 1^{\otimes (a-1)}\otimes r_{[1]i}\otimes 1^{\otimes (b-1)}\otimes r_{[2]i}\otimes 1^{\otimes (m-b)}\in \operatorname{End}V^{\otimes m}.
\end{gather*}
For example, $P_{12}=P\otimes 1$ when $m=3$.

\vskip 0.3cm

{\bf Acknowledgments}
\vspace{0.2em}

The second author is partially supported by the National Natural Science Foundation of China (Nos. 12071026, 12471025), supported by the Anhui Provincial Natural Science Foundation 2308085MA01. H. Zhang is partially supported by the National Natural Science Foundation of China (No. 12271332), and the Natural Science Foundation of Shanghai 22ZR1424600.

\section{Super Yangian}\label{se:Superyangian}

In the initial section of this paper, we recall several fundamental concepts, properties, and results pertinent to the super Yangian $\mathrm{Y}\left(\mathfrak{gl}_{M|N}\right)$.

\subsection{Definition of the super Yangian and its graded algebra}\label{se:Superyangian:1}

\begin{definition}
    The super Yangian $\mathrm{Y}\left(\mathfrak{gl}_{M|N}\right)$ is the complex associative superalgebra generated by elements $t_{ij}^{(r)}$ for $i,j=1,\ldots, M+N, r\in\mathbb{N}$. The parity is determined by setting
$|t_{ij}^{(r)}|=|i|+|j|$ for all $r\geqslant 1$. The defining relation in terms of the generated matrix
\begin{gather*}
T(u)=\sum_{i,j\in I}E_{ij}\otimes t_{ij}(u),\quad \text{with}\quad t_{ij}(u)=\delta_{ij}+\sum_{r\geqslant 1}t_{ij}^{(r)}u^{-r}
\end{gather*}
is given by
\begin{gather}\label{Y1}
R_{12}(u-v)T_1(u)T_2(v)=T_2(v)T_1(u)R_{12}(u-v),
\end{gather}
where $R(u)=1-u^{-1}P$ is the rational $R$-matrix, as a $\mathbb{Z}_2$-graded solution of the following Yang-Baxter equation
\begin{gather}\label{YB-1}
R_{12}(u)R_{13}(u+v)R_{23}(v)=R_{23}(v)R_{13}(u+v)R_{12}(u).
\end{gather}
\end{definition}

The super Yangian $\mathrm{Y}\left(\mathfrak{gl}_{M|N}\right)$ is a Hopf superalgebra with the comultiplication $\Delta$, the counit $\varepsilon$  and the antipode $S$ given by:
\begin{equation*}
\Delta(T(u))=T(u)\otimes T(u),\qquad \varepsilon(T(u))=1,\qquad S(T(u))=T(u)^{-1}.
\end{equation*}
The comultiplication can be explicitly rewritten in terms of generating series as follows:
\begin{gather*}
   \Delta( t_{ij}(u))= \sum_{k\in I}(-1)^{(|i|+|k|)(|k|+|j|)}t_{ik}(u)\otimes t_{kj}(u).
\end{gather*}

The generator matrix $T(u)$ possesses a Gauss decomposition with the unique series
\begin{gather*}
d_i(u),\quad e_{ij}(u),\quad f_{ji}(u)\in\mathrm{Y}\left(\mathfrak{gl}_{M|N}\right)\left[\left[u^{-1}\right]\right]
\end{gather*}
such that
\begin{align*}
T(u)=\left(1\otimes 1+\sum_{i<j}E_{ji}\otimes f_{ji}(u)\right)\left(\sum_{i\in I}E_{ii}\otimes d_i(u)\right)\left(1\otimes 1+\sum_{i<j}E_{ij}\otimes e_{ij}(u) \right)
\end{align*}
More specifically,
\begin{align*}
    t_{ii}(u)&=d_i(u)+\sum_{k<i}f_{ik}(u)d_k(u)e_{ki}(u), \\
    t_{ij}(u)&=d_i(u)e_{ij}(u)+\sum_{k<i}f_{ik}(u)d_k(u)e_{kj}(u), \\
    t_{ji}(u)&=f_{ji}(u)d_i(u)+\sum_{k<i}f_{jk}(u)d_k(u)e_{ki}(u),
\end{align*}
where $1\leqslant i<j\leqslant M+N$. Note that all series $d_i(u)$ are invertible.

It is obvious that $\left(T(u)^{-1}\right)^t=\left(T(u)^t\right)^{-1}$, where $t$ is always represented as the graded transposition in this article.
For simplicity, we utilize the following notations:
\begin{gather*}
    \widetilde{T}(u) =\sum_{i,j\in I}E_{ij}\otimes \tilde{t}_{ij}(u):=T(u)^{-1},\quad
    T^{\ast}(u)=\sum_{i,j\in I}E_{ij}\otimes t_{ij}^{\ast}(u):=\widetilde{T}^{t}(u).
\end{gather*}

There are several distinguished automorphisms and anti-automorphisms of the super Yangian $\mathrm{Y}\left(\mathfrak{gl}_{M|N}\right)$, and the specifics of the proof can be found in \cite[Section 2]{Na20}. Let $f(u)$ be any formal power series in $u^{-1}$ with the leading term $1$,
\begin{gather*}
    f(u)=1+f^{(1)}u^{-1}+f^{(2)}u^{-2}+\cdots \in\mathbb{C}\left[\left[u^{-1}\right]\right],
\end{gather*}
let $\lambda$ be any complex number, and let $B$ be any non-singular complex $(M+N)\times(M+N)$-matrix.
\begin{lemma}
Each of the mappings
\begin{align}
\mu_f:\ &T(u)\mapsto f(u)T(u), \label{auto:Ygl1}\\
&T(u)\mapsto T(u-\lambda), \label{auto:Ygl2}\\
&T(u)\mapsto BT(u)B^{-1}\label{auto:Ygl3}
\end{align}
defines an automorphism of $\mathrm{Y}\left(\mathfrak{gl}_{M|N}\right)$.
\end{lemma}
\begin{lemma}
Each of the mappings
\begin{align}
&\vartheta :\ T(u)\mapsto T(-u), \label{antiauto:Ygl1}\\
&\tau:\ T(u)\mapsto T^{t}(u), \label{antiauto:Ygl2}\\
&S:\ T(u)\mapsto \widetilde{T}(u)\label{antiauto:Ygl3}
\end{align}
defines an anti-automorphism of $\mathrm{Y}\left(\mathfrak{gl}_{M|N}\right)$.
\end{lemma}
The involutive anti-automorphisms \eqref{antiauto:Ygl1} and \eqref{antiauto:Ygl2} commute with each other. Thus, their composition
\begin{gather}\label{auto:Ygl4}
    \rho:\ T(u)\mapsto T^{\mathrm{t}}(-u)
\end{gather}
is an involutive automorphism.

Let $$T^{\ast}(u):=(1\otimes \tau)\widetilde{T}(u)=(\tau\otimes 1)\widetilde{T}(u),$$
where $\tau$\footnote{Here, we utilize the same symbols $\tau$ as in equation \eqref{antiauto:Ygl2}. This will not result in any confusion since $T^{t}(u)$ is equal to $(\tau\otimes1)T(u)$.} is an anti-superinvolution of $\operatorname{End}V$ defined by
 $\tau\left(E_{ij}\right)=(-1)^{|i||j|+|i|}E_{ji}$.
Then the mapping
\begin{gather}\label{auto:Ygl5}
    \psi:\ T(u)\mapsto T^{\ast}(u)
\end{gather}
 is an automorphism of $\mathrm{Y}\left(\mathfrak{gl}_{M|N}\right)$ but not involutive, since the antipode $S$ does not.

We recall the two different ascending $\mathbb{Z}$-filtrations of $\mathrm{Y}\left(\mathfrak{gl}_{M|N}\right)$ defined by Nazarov in \cite{Na20}. The first assigns the generator $t_{ij}^{(r)}$  of degree $r$ for each $r\geqslant 1$. The associated graded algebra, denoted by $\operatorname{gr}_0\mathrm{Y}\left(\mathfrak{gl}_{M|N}\right)$, is supercommutative and is generated freely by the elements $t_{ij}^{(r)}$ (with no confusion arising from the same notations), where $i,j= 1,\ldots,n$ and $r\in\mathbb{N}$. The second ascending $\mathbb{Z}$-filtration is of particular importance, and we now provide further details below.

Define $\deg t_{ij}^{(r)}=r-1$ for all $i,j$, and set $\deg (xy)=\deg x+\deg y$ for homogeneous elements $x,y\in\mathrm{Y}\left(\mathfrak{gl}_{M|N}\right)$. Then,
\begin{gather*}
   \mathcal{F}_0\subset \mathcal{F}_1\subset \cdots \subset \mathcal{F}_p\subset \cdots,
\end{gather*}
is a $\mathbb{Z}$-filtration of $\mathrm{Y}(\mathfrak{gl}_{M|N}) $,
where
\begin{gather*}
    \mathcal{F}_p:=\left\{\, x\in\mathrm{Y}\left(\mathfrak{gl}_{M|N}\right) \,|\, \deg x\leqslant p\,\right\},\quad p>0, \text{ and } \mathcal{F}_0=\mathbb{C}.
\end{gather*}
The associated graded algebra is denoted by
\begin{gather*}
    \operatorname{gr}\mathrm{Y}\left(\mathfrak{gl}_{M|N}\right):=\mathcal{F}_0\oplus \left(\mathop{\bigoplus}\limits_{p=1}^{\infty}\mathcal{F}_{p}/\mathcal{F}_{p-1}\right).
\end{gather*}

Set $\mathfrak{gl}_{M|N}[x]=\mathfrak{gl}_{M|N}\otimes \mathbb{C}[x]$.
\begin{proposition}\label{grU}\cite[Theorem 2.8]{Na20}
There exists an isomorphism of superalgebras between the filtered superalgebra $\operatorname{gr}\mathrm{Y}\left(\mathfrak{gl}_{M|N}\right)$ and the universal enveloping superalgebra $\mathrm{U}\left(\mathfrak{gl}_{M|N}[x]\right)$ such that
\begin{gather*}
\bar{t}_{ij}^{(r+1)}\mapsto -(-1)^{|j|}E_{ji}x^{r},\quad i,j\in I,\ r\in\mathbb{Z}_+,
\end{gather*}
where $\bar{t}_{ij}^{(r+1)}$ denotes the image of $t_{ij}^{(r+1)}$ in the component $\mathcal{F}_{r}/\mathcal{F}_{r-1}$, for $r\geqslant 1$.
\end{proposition}
Thus, the superalgebra $\mathrm{Y}\big(\mathfrak{gl}_{M|N}\big)$ is a deformation of the universal enveloping superalgebra $\mathrm{U}\big(\mathfrak{gl}_{M|N}[x]\big)$ in the class of Hopf superalgebras.

\subsection{The center of super Yangian}
The center of an (associative) superalgebra $\mathcal{A}$ is an ideal consisting of elements which super-commute with $\mathcal{A}$, denoted by $\mathcal{Z}\mathcal{A}$. In this subsection, we briefly review some known results related to $\mathcal{Z}\mathrm{Y}\left(\mathfrak{gl}_{M|N}\right)$, for more details, please refer to \cite{Na20}.

\begin{proposition}\label{centre:gl}\cite[Proposition 3.1]{Na20}
There is a series
\begin{gather*}
z(u)=1+z^{(1)}u^{-1}+z^{(2)}u^{-2}+z^{(3)}u^{-3}+\cdots \in\mathcal{Z}\mathrm{Y}\left(\mathfrak{gl}_{M|N}\right)
\end{gather*}
such that
\begin{align}\label{centre:1}
&\sum_k \tilde{t}_{kj}(u)t_{ik}(u+M-N)=\delta_{ij}z(u), \\ \label{centre:2}
&\sum_k t_{kj}(u+M-N)\tilde{t}_{ik}(u)=\delta_{ij}z(u).
\end{align}
Moreover, $z^{(1)}=0$ and $z^{(2)},z^{(3)},\ldots$ are free generators of $\mathcal{Z}\mathrm{Y}\left(\mathfrak{gl}_{M|N}\right)$.
\end{proposition}

\begin{corollary}\cite[Proposition 3.4]{Na20}
    The series $z(u)$ is invariant under the action of $\tau$ in \eqref{antiauto:Ygl2}.
\end{corollary}
\begin{corollary}\cite[Corollary 3.5]{Na20}
    The anti-automorphism $S$ in \eqref{antiauto:Ygl3} maps $z(u)$ onto $z(u)^{-1}$.
\end{corollary}

\begin{corollary}\cite[Proposition 3.7]{Na20}\label{Na:Z:U}
    The image of elements $z^{(r)}$ for $r\geqslant 2$ in the graded algebra $\operatorname{gr}\mathrm{Y}\left(\mathfrak{gl}_{M|N}\right)$ is
    \begin{gather*}
    (1-r)\sum_{i\in I}(-1)^{|i|}\bar{t}_{ii}^{(r-1)}
    \end{gather*}
    with degree $\deg z^{(r)}=r-2$. Here $\bar{t}_{ij}^{(r)}$ denotes the image of $t_{ij}^{(r)}$ in the $r$-th component of $\operatorname{gr}\mathrm{Y}\big(\mathfrak{gl}_{M|N}\big)$.
\end{corollary}

The \textit{quantum Berezinian} was introduced by Nazarov in \cite{Na91} as the following power series with coefficients in the Yangian $\mathrm{Y}\left(\mathfrak{gl}_{M|N}\right)$
\begin{align*}
\mathfrak{B}_{M|N}(u)=&\sum_{\sigma\in \mathfrak{S}_M}\operatorname{sgn}\sigma\cdot t_{\sigma(1),1}(u+M-N-1)t_{\sigma(2),2}(u+M-N-2)
           \cdots t_{\sigma(M),M}(u-N) \\
&\times \sum_{\sigma\in \mathfrak{S}_N} \operatorname{sgn}\sigma \cdot \tilde{t}_{M+1,M+\sigma(1)}(u-N)\tilde{t}_{M+2,M+\sigma(2)}(u-N+1)
          \cdots \tilde{t}_{M+N,M+\sigma(N)}(u-1),
\end{align*}
where $\operatorname{sgn}\sigma$ denotes the signature of the permutation $\sigma$ of the symmetric group $\mathfrak{S}_m$, see \cite[Section 4]{Na20}. The coefficients of the quantum Berezinian generate the center of the Yangian, conjectured in \cite{Na91} and proved by Gow in \cite{Go05}.
The following proposition is a super-version of the quantum Liouville formula.
\begin{proposition}\cite[Theorem 4.1]{Na20}
We have the equality in $\mathrm{Y}(\mathfrak{gl}_{M|N})\left[\left[ u^{-1}\right]\right]$,
\begin{gather}\label{qLiou-gl}
z(u)=\frac{\mathfrak{B}_{M|N}(u+1)}{\mathfrak{B}_{M|N}(u)}.
\end{gather}
\end{proposition}

\vspace{1em}
Recall that the special linear Lie superalgebra $\mathfrak{sl}_{M|N}$ is a Lie sub-superalgebra of $\mathfrak{gl}_{M|N}$ consisting of all matrices with zero supertrace.

\begin{definition}\cite{Go07,Ts20}
    The special super Yangian $\mathrm{Y}\left(\mathfrak{sl}_{M|N}\right)$ is the sub-superalgebra of $\mathrm{Y}\left(\mathfrak{gl}_{M|N}\right)$ consisting of elements that are invariant under all automorphisms \eqref{auto:Ygl1}. Specifically, we have
    $$\mathrm{Y}(\mathfrak{sl}_{M|N}):=\left\{y\in\mathrm{Y}(\mathfrak{gl}_{M|N})|\ \mu_f(y)=y \text{ for all }f\right\}.$$
\end{definition}

The Lie superalgebra $\mathfrak{sl}_{N|N}$ for $M=N$ is not simple, as it has a one-dimensional ideal consisting of scalar matrices. Therefore, we only show the results of $\mathrm{Y}\left(\mathfrak{sl}_{M|N}\right)$ for $M\neq N$.

\begin{proposition}\cite[Proposition 3]{Go07}
    For $M\neq N$, the superalgebra $\mathrm{Y}\left(\mathfrak{gl}_{M|N}\right)$ has the following decomposition
    \begin{gather*}
        \mathrm{Y}\left(\mathfrak{gl}_{M|N}\right)=\mathcal{Z}\mathrm{Y}\left(\mathfrak{gl}_{M|N}\right)\otimes \mathrm{Y}\left(\mathfrak{sl}_{M|N}\right).
    \end{gather*}
\end{proposition}

\begin{corollary}
    For $M\neq N$, the superalgebra $\mathrm{Y}\left(\mathfrak{sl}_{M|N}\right)$ is isomorphic to the quotient of $\mathrm{Y}\left(\mathfrak{gl}_{M|N}\right)$ by the equation $\mathfrak{B}_{M|N}(u)=1$.
\end{corollary}

\begin{corollary}
    For $M\neq N$,  the superalgebra $\mathrm{Y}\left(\mathfrak{sl}_{M|N}\right)$ has a Hopf superalgebra structure inherited from $\mathrm{Y}\left(\mathfrak{gl}_{M|N}\right)$.
\end{corollary}

\section{Twisted super Yangian}\label{se:twistedsuperyangian}
In this section, we establish an equivalence between the presentation of the twisted super Yangian $\mathrm{Y}^{tw}\left(\mathfrak{osp}_{M|2n}\right)$ obtained by Molev in \cite[Section 2.16, Example 16]{Mo07} and the presentation introduced by Briot and Ragoucy in 2003 \cite[Theorem 3.2]{BR03} subject to the restriction of the special matrix $\mathcal{G}$.
Given the parity function $|\cdot|$ on the set $I=\{1,\ldots,M+N\}$ such that
\begin{equation*}
|i|=\begin{cases}
0,  &\text{if}\ i\in I_+, \\
1,  &\text{if}\ i\in I_-,
\end{cases}
\end{equation*}
where $I_+=\{1,\ldots,M\}$ and  $I_-=\{M+1,\ldots,M+N\}$.
For convenience, from now on we always put $N=2n$.

\subsection{Twisted super Yangian $\mathrm{Y}^{tw}\left(\mathfrak{osp}_{M|N}\right)$}

Set $J=\sum_{i\in I}(-1)^{|i|}E_{ii}$.
Let
\begin{gather*}
\mathcal{G}=\left(g_{ij}\right)_{i,j\in I}=\begin{pmatrix}
\mathcal{G}_{(1)} & 0                      \\
0                   & \mathcal{G}_{(2)}
\end{pmatrix}
\end{gather*}
be a non-degenerate $\mathbb{Z}_2$-graded matrix such that $\mathcal{G}_{(1)}$ is the $(M\times M)$-symmetric matrix and $\mathcal{G}_{(2)}$ is the $(N\times N)$-antisymmetric matrix. Clearly,
$\mathcal{G}^t=J\mathcal{G}$. We call a matrix that has this form a (nondegenerate) supersymmetric block matrix.

Consider the formal power series corresponding to $\mathcal{G}$:
\begin{gather*}
s_{ij}(u)=g_{ij}+s_{ij}^{(1)}u^{-1}+s_{ij}^{(2)}u^{-2}+\cdots
\end{gather*}
and set $S(u)=\sum_{i,j\in I}E_{ij}\otimes s_{ij}(u)$.
\begin{definition}
The twisted super Yangian $\mathrm{Y}^{tw}\left(\mathfrak{osp}_{M|N}\right)$ is an associative superalgebra generated by elements $s_{ij}^{(r)}$ for $i,j\in I$ and $r\in\mathbb{N}$, with $\mathbb{Z}_2$-grading $|s_{ij}^{(r)}|=|i|+|j|$ for all $r\geqslant 1$, and subject to the following relations
\begin{align}\label{Ytw1}
R_{12}(u-v)S_1(u)R_{12}^{t_1}(-u-v)S_2(v)&=S_2(v)R_{12}^{t_2}(-u-v)S_1(u)R_{12}(u-v), \\ \label{Ytw2}
J\,S^t(-u)&=S(u)+\frac{S(u)-S(-u)}{2u}.
\end{align}
\end{definition}

\begin{remark}
  For $M=0$ or $N=0$, the twisted super Yangian $\mathrm{Y}^{tw}\left(\mathfrak{osp}_{M|N}\right)$ can specialize to the symplectic twisted Yangian $\mathrm{Y}^{tw}\left(\mathfrak{sp}_{N}\right)$ or the orthogonal twisted Yangian $\mathrm{Y}^{tw}\left(\mathfrak{o}_M\right)$ \cite[Chapter 2]{Mo07}, respectively.
\end{remark}

\begin{proposition}\label{Ytw:eqv}
    Denote the twisted Yangian associated with $\mathcal{G}$ by $\mathrm{Y}_{\mathcal{G}}^{tw}\left(\mathfrak{osp}_{M|N}\right)$. Then the superalgebras $\mathrm{Y}_{\mathcal{G}}^{tw}\left(\mathfrak{osp}_{M|N}\right)$ are pairwise isomorphic. In other words, if $\mathcal{G}'=B\mathcal{G}B^t$ for a nonsingular matrix $B$ with zero in the upper-right $(M\times N)$-block and lower-left $(N\times M)$-block, there exists an isomorphism of superalgebras $\mathrm{Y}_{\mathcal{G}'}^{tw}\left(\mathfrak{osp}_{M|N}\right)$ and $\mathrm{Y}_{\mathcal{G}}^{tw}\left(\mathfrak{osp}_{M|N}\right)$ such that
    \begin{gather}\label{Ytw:eqv:G}
        S(u)\mapsto BS(u)B^t.
    \end{gather}
\end{proposition}

\begin{proof}
    Applying the partial transposition $t_1$ to both sides of
    \begin{gather*}
        R_{12}(-u-v)B_1B_2=B_2B_1R_{12}(-u-v),
    \end{gather*}
    one gets
    \begin{gather*}
        B_1^tR^{t_1}_{12}(-u-v)B_2=B_2R^{t_1}_{12}(-u-v)B_1^t.
    \end{gather*}
    It follows that
    \begin{align*}
        R_{12}(u-v)B_1S_1(u)B_1^tR_{12}^{t_1}(-u-v)B_2S_2(v)B_2^t
        &=B_2B_1R_{12}(u-v)S_1(u)R_{12}^{t_1}(-u-v)S_2(v)B_1^tB_2^t \\
        &=B_2B_1S_2(v)R_{12}^{t_2}(-u-v)S_1(u)R_{12}(u-v)B_1^tB_2^t \\
        &=B_2S_2(v)B_2^tR_{12}^{t_2}(-u-v)B_1S_1(v)B_1^tR_{12}(u-v).
    \end{align*}
    Moreover, we also have
    \begin{gather*}
        J\left(BS(u)B^t\right)^t=BJS^t(u)B^t,
    \end{gather*}
    which preserves \eqref{Ytw2}. Since $B$ is invertible, the homomorphism \eqref{Ytw:eqv:G} is bijective, completing the proof.
    
\end{proof}

\begin{proposition}\label{prop:embed}
The assignment
\begin{gather}\label{embedding}
\jmath:\ S(u)\mapsto T(u)\mathcal{G}T^t(-u)
\end{gather}
defines an embedding of superalgebras $\mathrm{Y}^{tw}\left(\mathfrak{osp}_{M|N}\right)$ and $\mathrm{Y}\left(\mathfrak{gl}_{M|N}\right)$.
\end{proposition}
\begin{proof}
Since
\begin{gather*}
PP^{t_1}=J_1P^{t_1}=P^{t_2}P,\quad P\mathcal{G}_1\mathcal{G}_2=\mathcal{G}_2\mathcal{G}_1P,
\end{gather*}
we have
\begin{gather*}
    \mathcal{G}_1P^{t_1}\mathcal{G}_2=\mathcal{G}_2P^{t_2}\mathcal{G}_1,\quad P\mathcal{G}_1P^{t_1}\mathcal{G}_2=\mathcal{G}_2P^{t_2}\mathcal{G}_1P.
\end{gather*}
It implies the equation
\begin{gather*}
R_{12}(u-v)\mathcal{G}_1R_{12}^{t_1}(-u-v)\mathcal{G}_2
     =\mathcal{G}_2R_{12}^{t_2}(-u-v)\mathcal{G}_1R_{12}(u-v)
\end{gather*}
holds.
Consequently, we have
\begin{align*}
&\quad R_{12}(u-v)T_1(u)\mathcal{G}_1 T_1^t(-u)R_{12}(-u-v)^{t_1}T_2(v)\mathcal{G}_2T_2^t(-v) \\
&=R_{12}(u-v)T_1(u)\mathcal{G}_1 T_2(v) R_{12}^{t_1}(-u-v)T_1^t(-u) \mathcal{G}_2 T_2^t(-v) \\
&=T_2(v)T_1(u)R_{12}(u-v)\mathcal{G}_1 R_{12}^{t_1}(-u-v)\mathcal{G}_2T^t_1(-u)T^t_2(-v) \\
&=T_2(v)T_1(u)\mathcal{G}_2 R_{12}^{t_2}(-u-v)\mathcal{G}_1R_{12}(u-v)T_1^t(-u)T_2^t(-v) \\
&=T_2(v)\mathcal{G}_2 T_1(u) R_{12}^{t_2}(-u-v)\mathcal{G}_1T_2^t(-v)T_1^t(-u)R_{12}(u-v) \\
&=T_2(v)\mathcal{G}_2  T_2^t(-v) R_{12}^{t_2}(-u-v)T_1(u)\mathcal{G}_1T_1^t(-u)R_{12}(u-v).
\end{align*}

Furthermore, the $(i,j)$-position of matrix $X(u):=T(u)\mathcal{G}T^t(-u)$ is
\begin{gather*}
\sum\limits_{a,b\in I}(-1)^{|i||j|+|i||a|+|j|+|a|}g_{ab}t_{ia}(u)t_{jb}(-u).
\end{gather*}
Then
\begin{align*}
J X^t(-u)-X(u)
        &=\sum_{i,j,a,b}\left( (-1)^{|j||a|+|j|+|a|}g_{ab}t_{ja}(-u)t_{ib}(u)
                  -(-1)^{|i||j|+|i||a|+|j|+|a|}g_{ba}t_{ib}(u)t_{ja}(-u)\right) \\
        &=\sum_{i,j,a,b}(-1)^{|j||a|+|j|+|a|}g_{ab}\,[t_{ja}(-u),\,t_{ib}(u)]  \\
        &=\frac{1}{2u}\sum_{i,j,a,b}(-1)^{|i||j|+|i||a|+|j|+|a|}g_{ab}\left( t_{ia}(u)t_{jb}(-u)-t_{ia}(u)t_{jb}(-u)\right) \\
        &=\frac{1}{2u}\left(X(u)-X(-u)\right),
\end{align*}
owing to $g_{ab}=(-1)^{|a|}g_{ba}$ and $g_{ab}=0$ for $|a|\neq |b|$.

Next, similar to the super Yangian $\mathrm{Y}\big(\mathfrak{gl}_{M|N}\big)$, we define the first ascending $\mathbb{Z}$-filtration of the twisted super Yangian $\mathrm{Y}^{tw}\left(\mathfrak{osp}_{M|N}\right)$ such that $\deg s_{ij}^{(r)}=r$ with the graded algebra $\operatorname{gr}_0\mathrm{Y}^{tw}\left(\mathfrak{osp}_{M|N}\right)$. We will also employ the second $\mathbb{Z}$-filtration in Section \ref{se:Invo-graded}. The homomorphism \eqref{embedding} preserves the filtration, hence, it induces a superalgebraic homomorphism
$$\bar{\jmath}:\ \operatorname{gr}_0\mathrm{Y}^{tw}\left(\mathfrak{osp}_{M|N}\right)\rightarrow \operatorname{gr}_0\mathrm{Y}\left(\mathfrak{gl}_{M|N}\right).$$
By the same arguments as \cite[Theorem 2.4.3]{Mo07}, we deduce that $\bar{\jmath}$ is injective. Thus, $\jmath$ is an embedding.

\end{proof}

\begin{corollary}\label{co:coideal}
    The superalgebra $\mathrm{Y}^{tw}\left(\mathfrak{osp}_{M|N}\right)$ is a left coideal of $\mathrm{Y}\left(\mathfrak{gl}_{M|N}\right)$, that is,
    \begin{gather*}
        \Delta\left( \mathrm{Y}^{tw}\left(\mathfrak{osp}_{M|N}\right) \right)\subset \mathrm{Y}\left(\mathfrak{gl}_{M|N}\right) \otimes \mathrm{Y}^{tw}\left(\mathfrak{osp}_{M|N}\right).
    \end{gather*}
    The generators of $\mathrm{Y}^{tw}\left(\mathfrak{osp}_{M|N}\right)$ under the comultiplication have the form
    \begin{gather*}
        \Delta(s_{ij}(u))=\sum\limits_{a,b\in I}(-1)^{|i||j|+|a||j|+|i||a|+|j||b|+|j|+|a|}t_{ia}(u)t_{jb}(-u)\otimes s_{ab}(u).
    \end{gather*}
\end{corollary}
The Proposition \ref{prop:embed} and Corollary \ref{co:coideal} imply that the twisted super Yangian $\mathrm{Y}^{tw}\left(\mathfrak{osp}_{M|N}\right)$ is a coideal sub-superalgebra of the super Yangian $\mathrm{Y}\left(\mathfrak{gl}_{M|N}\right)$.

Set
\begin{gather*}
\widetilde{S}(u)=\sum_{i,j\in I}E_{ij}\otimes \tilde{s}_{ij}(u):=S(u)^{-1}.
\end{gather*}
Multiplying $\widetilde{S}_2(u)$ on both sides of relation \eqref{Ytw1},
one gets
\begin{gather*}
R_{12}^{t_2}(-u-v)S_1(u)R_{12}(u-v)\widetilde{S}_2(v)
=\widetilde{S}_2(v)R_{12}(u-v)S_1(u)R_{12}^{t_1}(-u-v).
\end{gather*}
More precisely,
\begin{equation*}
\begin{split}
&s_{ij}(u)\tilde{s}_{kl}(v)-\varsigma_{i,j;k,l}\,\tilde{s}_{kl}(v)s_{ij}(u) \\
&=\frac{\varsigma_{i,j;k,l}}{u-v}\left(\delta_{jk}\sum_b(-1)^{|j|}\varsigma_{i,b;b,l}s_{ib}(u)\tilde{s}_{bl}(v)
  -\delta_{il}\sum_a (-1)^{|a|}\varsigma_{i,a;k,j}\tilde{s}_{ka}(v)s_{aj}(u)\right) \\
&+\frac{\varsigma_{i,j;k,l}}{u+v}\left(\delta_{jl}\sum_a(-1)^{|j|+|j||k|+|k||a|}\tilde{s}_{ka}(v)s_{ia}(u)
-\delta_{ik}\sum_a(-1)^{|i||j|+|a||l|+|j||l|}s_{aj}(u)\tilde{s}_{al}(v)\right) \\
&+\frac{1}{u^2-v^2}\left(\delta_{ik}\sum_as_{ja}(u)\tilde{s}_{al}(v)(-1)^{|i||j|+|j|}\varsigma_{j,a;a,l}\varsigma_{i,j;k,l}
-\delta_{jl}\sum_b\tilde{s}_{kb}(v)s_{bi}(u)(-1)^{|i||j|+|i|}\varsigma_{ia;ak}\right),
\end{split}
\end{equation*}
where the notation $\varsigma_{ab;cd}$ means $(-1)^{(|a|+|b|)(|c|+|d|)}$. Consequently, the following lemma is immediately obtained.
\begin{lemma}
If $\{i,j\}\cap\{k,l\}=\varnothing$, then
\begin{gather*}
\left[s_{ij}(u),\,\tilde{s}_{kl}(v)\right]=0.
\end{gather*}
In particular, the coefficients of $s_{ij}(u)$ and $\tilde{s}_{kl}(v)$ commute for $i,j\leqslant M<k,l$.
\end{lemma}

Furthermore, we can also multiply both sides of \eqref{Ytw1} by $R_{12}(u-v)^{-1}$, $\widetilde{S}_1(u)$, $\big(R_{12}^{t_1}(-u-v)\big)^{-1}$ (resp. $\big(R_{12}^{t_2}(-u-v)\big)^{-1}$), $\widetilde{S}_2(v)$ successively due to
\begin{gather*}
    R(u)^{-1}=\frac{u^2}{u^2-1}R(-u),\qquad \big(R^{t_i}(u)\big)^{-1}=R^{t_i}(-u+M-N)~~\text{ for }~~i=1,2.
\end{gather*}
That is to say,
\begin{lemma}The following equation holds in $\mathrm{End} V^{\otimes 2}\otimes\mathrm{Y}^{tw}\big(\mathfrak{osp}_{M|N}\big)\big[\big[u^{-1},v^{-1}\big]\big]$,
    \begin{equation}\label{Ytw3}
\begin{split}
&R_{12}(u-v)\widetilde{S}_1(-u-\frac{M-N}{2})R_{12}^{t_2}(-u-v)\widetilde{S}_2(-v-\frac{M-N}{2}) \\
   =&\widetilde{S}_2(-v-\frac{M-N}{2})R_{12}^{t_1}(-u-v)\widetilde{S}_1(-u-\frac{M-N}{2})R_{12}(u-v).
\end{split}
\end{equation}
\end{lemma}

\subsection{An equivalent presentation for twisted super Yangian}
This subsection is devoted to obtaining an alternative presentation for twisted super Yangian, which is inspired by Molev's book in \cite[Section 2.15]{Mo07}. For any supersymmetrix block matrix $\mathcal{G}$, define a transposition $\iota$ on a $\mathbb{Z}_2$-graded matrix $X$ associated with $\mathcal{G}$ by
\begin{gather*}
X^{\iota}=\mathcal{G}X^t\mathcal{G}^{-1}.
\end{gather*}
Since $J^2=1$, $\mathcal{G}^t$ is invertible. It follows that
\begin{gather*}
(X^{\iota})^{\iota}=\mathcal{G}\left(\mathcal{G}X^t\mathcal{G}^{-1}\right)^t\mathcal{G}^{-1}
=\mathcal{G}(\mathcal{G}^{-1})^t(X^t)^t\mathcal{G}^t\mathcal{G}^{-1}
=J(X^t)^tJ=X.
\end{gather*}

Consider another formal power series
\begin{gather*}
\mathfrak{s}_{ij}(u)=\delta_{ij}+\mathfrak{s}_{ij}^{(1)}u^{-1}+\mathfrak{s}_{ij}^{(2)}u^{-2}+\cdots ,
\end{gather*}
and $\mathcal{S}(u)=\sum\limits_{i,j\in I}E_{ij}\otimes \mathfrak{s}_{ij}(u)$.
\begin{proposition}\label{prop:alterpresentation}
The twisted super Yangian $\mathrm{Y}^{tw}\left(\mathfrak{osp}_{M|N}\right)$ is isomorphic to the superalgebra $\mathcal{Y}^{tw}\left(\mathfrak{osp}_{M|N}\right)$ generated by elements $\mathfrak{s}_{ij}^{(r)}$, with parity $|\mathfrak{s}_{ij}^{(r)}|=|i|+|j|$, for $i,j\in I$ and $r\in\mathbb{N}$ with the defining relations
\begin{align}\label{Yeqv1}
R_{12}(u-v)\mathcal{S}_1(u)R_{12}^{\iota_1}(-u-v)\mathcal{S}_2(v)
    &=\mathcal{S}_2(v)R_{12}^{\iota_1}(-u-v)\mathcal{S}_1(u)R_{12}(u-v), \\ \label{Yeqv2}
\mathcal{S}^{\iota}(-u)&=\mathcal{S}(u)+\frac{\mathcal{S}(u)-\mathcal{S}(-u)}{2u}.
\end{align}
More concretely,
the assignment
\begin{gather}\label{Yeqv3}
\phi:\ S(u)\mapsto \mathcal{S}(u)\mathcal{G}
\end{gather}
defines an isomorphism between the superalgebra $\mathrm{Y}^{tw}\left(\mathfrak{osp}_{M|2n}\right)$ and $\mathcal{Y}^{tw}\left(\mathfrak{osp}_{M|2n}\right)$.
\end{proposition}

\begin{proof}
    Given a nondegenerate matrix $\mathcal{G}$, it follows that the mapping $\phi$ is a bijection.
    Since $P_{12}\mathcal{G}_1\mathcal{G}_2=\mathcal{G}_2\mathcal{G}_1P_{12}$ and \eqref{Yeqv1}, we have
    \begin{align*}
        R_{12}(u-v)\mathcal{S}_1(u)\mathcal{G}_1R_{12}^{t_1}(-u-v)\mathcal{S}_2(v)\mathcal{G}_2
        &=R_{12}(u-v)\mathcal{S}_1(u)R_{12}^{\iota_1}(-u-v)\mathcal{S}_2(v)\mathcal{G}_1\mathcal{G}_2 \\
        &=\mathcal{S}_2(v)R_{12}^{\iota_2}(-u-v)\mathcal{S}_1(u)R_{12}(u-v)\mathcal{G}_1\mathcal{G}_2 \\
        &=\mathcal{S}_2(v)R_{12}^{\iota_2}(-u-v)\mathcal{G}_2\mathcal{S}_1(u)\mathcal{G}_1R_{12}(u-v) \\
        &=\mathcal{S}_2(v)\mathcal{G}_2R_{12}^{t_2}(-u-v)\mathcal{S}_1(u)\mathcal{G}_1R_{12}(u-v).
    \end{align*}
    Moreover, by \eqref{Yeqv2}, we also have
    \begin{align*}
        J\left(\mathcal{S}(-u)\mathcal{G}\right)^t&=J\mathcal{G}^t\mathcal{S}^t(-u)=\mathcal{G}\mathcal{S}^t(-u)
        =\mathcal{S}^{\iota}(-u)\mathcal{G} \\
        &=\mathcal{S}(u)\mathcal{G}+\frac{\mathcal{S}(u)\mathcal{G}-\mathcal{S}(-u)\mathcal{G}}{2u}.
    \end{align*}

\end{proof}

The isomorphism $\phi$ enables us to obtain some properties related to $\mathfrak{s}_{ij}(u)$.
\begin{corollary}
The coefficients of $\mathfrak{s}_{ij}(u)$ and $\tilde{\mathfrak{s}}_{kl}(v)$ commute for $i,j\leqslant M<k,l$.
\end{corollary}

Denote
\begin{gather*}
    \widetilde{\mathcal{S}}(u)=\sum_{i,j\in I}E_{ij}\otimes \tilde{\mathfrak{s}}_{ij}(u):=\mathcal{S}(u)^{-1}.
\end{gather*}

\begin{corollary}\label{S-inver}
The following equation holds in $\mathrm{End}V^{\otimes 2}\otimes\mathcal{Y}^{tw}\big(\mathfrak{osp}_{M|N}\big)\big[\big[u^{-1},v^{-1}\big]\big]$,
    \begin{equation}\label{Yeqv4}
\begin{split}
&R_{12}(u-v)\widetilde{S}_1\left(-u-\frac{M-N}{2}\right)R_{12}^{\iota_1}(-u-v)\widetilde{S}_2\left(-v-\frac{M-N}{2}\right) \\
   =&\widetilde{S}_2\left(-v-\frac{M-N}{2}\right)R_{12}^{\iota_1}(-u-v)\widetilde{S}_1\left(-u-\frac{M-N}{2}\right)R_{12}(u-v).
\end{split}
\end{equation}
\end{corollary}

Following the embedding \eqref{embedding}, one immediately gets:
\begin{proposition}\eqref{prop:embed}
The assignment
\begin{gather*}
\mathcal{S}(u)\mapsto T(u)T^{\iota}(-u)
\end{gather*}
defines an embedding of the superalgebras $\mathcal{Y}^{tw}\left(\mathfrak{osp}_{M|N}\right)$ and $\mathrm{Y}\left(\mathfrak{gl}_{M|N}\right)$.
\end{proposition}

\begin{remark}\label{BRtw}
Fix $\mathcal{G}=\mathcal{G}_0=\sum_{i\in I}(-1)^{|i|}\theta_i E_{ii'}$. If we define the signs $\theta_i$ and $i'$ on the index set $I$ as follows:
\begin{equation*}
\begin{split}
    \theta_i&=\begin{cases}
   1,  &1\leqslant i\leqslant M+n, \\
   -1,  &M+n+1\leqslant i\leqslant M+N,
\end{cases} \\
   i'&=\begin{cases}
M+1-i,  &1\leqslant i\leqslant M  \\
2M+N+1-i,  &M+1\leqslant i\leqslant M+N.
\end{cases}
\end{split}
\end{equation*}
 Then, the definition of $\mathcal{Y}^{tw}\left(\mathfrak{osp}_{M|2n}\right)$ coincides with the ``+'' version of the twisted super Yangians referred to \cite{BR03}. In this case, $\iota$ acts on a matrix $X=(X_{ij})$ by
\begin{gather*}
X^{\iota}=\sum_{i,j\in I}E_{ij}\otimes X_{ij}^{\iota}\quad \text{with}\quad x_{ij}^{\iota}=(-1)^{|i||j|+|i|}\theta_i\theta_jX_{j'i'}.
\end{gather*}
\end{remark}

\section{The center of twisted super Yangian}\label{se:Invo-graded}
In this section, our purpose is to describe the center of the twisted super Yangian $\mathrm{Y}^{tw}\left(\mathfrak{osp}_{M|N}\right)$, denoted by  $\mathcal{Z}\mathrm{Y}^{tw}\left(\mathfrak{osp}_{M|N}\right)$. Before that, we need to investigate the connection between $\mathrm{Y}^{tw}\left(\mathfrak{osp}_{M|N}\right)$ and a twisted polynomial current Lie superalgebra $\mathfrak{gl}_{M|N}[x]^{\theta}$.

\subsection{Twisted polynomial current Lie superalgebra}
Let $\mathfrak{a}$ be a finite-dimensional Lie superalgebra with a unity, and let $\theta$ be an involution of $\mathfrak{a}$ such that $\theta^2=1$. Define the Lie sub-superalgebras $\mathfrak{a}_+$ and $\mathfrak{a}_-$ of $\mathfrak{a}$ by\footnote{According to the description in \cite{SW24}, the involution of a Lie superalgebra $\mathfrak{a}$ is defined as an automorphism with order 2 or 4. In our paper, we focus exclusively on the case where $\mathfrak{a}^+$ forms a fixed sub-superalgebra.}
\begin{gather*}
\mathfrak{a}_+:=\{\,a\in\mathfrak{a}\,|\,\theta(a)=a\,\}\quad \text{and}\quad \mathfrak{a}_-:=\{\,a\in\mathfrak{a}\,|\,\theta(a)=-a\,\}.
\end{gather*}
It is evident that $\mathfrak{a}=\mathfrak{a}_+\oplus \mathfrak{a}_-$.
Let $\mathfrak{a}[x]^{\theta}$ be the \textit{twisted polynomial current Lie superalgebra}:
\begin{gather*}
    \mathfrak{a}[x]^{\theta}:=\{\,a(u)\in \mathfrak{a}[x]\,|\,\theta\left(a(u)\right)=a(-u)\,\}.
\end{gather*}
Equivalently,
\begin{gather*}
\mathfrak{a}[x]^{\theta}=\mathfrak{a}_+\oplus \mathfrak{a}_-x\oplus \mathfrak{a}_+x^2\oplus \mathfrak{a}_-x^3\oplus \cdots
\end{gather*}

\begin{lemma}\label{trivialc}
Suppose that the centralizer of $\mathfrak{a}_+$ in $\mathfrak{a}$ is trivial with respect to the standard supercommutator. In that case, the center of the universal enveloping superalgebra of $\mathfrak{a}[x]^{\theta}$ is also trivial.
\end{lemma}

\begin{proof}
Assume $\dim \mathfrak{a}_+=p$ and $\dim \mathfrak{a}_-=q$. Let $\{f_1^+,\ldots,f_p^+\}$ and $\{f_1^-,\ldots,f_q^-\}$ be the $\mathbb{Z}_2$-homogeneous bases of $\mathfrak{a}_+$ and $\mathfrak{a}_-$, respectively, such that
\begin{gather*}
[f_i^+,\,f_j^-]=\sum_{k=1}^{q}c_{ij}^kf_k^-\in\mathfrak{a}_-, \text{ with }c_{ij}^k\in\mathbb{C},\quad 1\leqslant i\leqslant p,\ 1\leqslant j\leqslant q,
\end{gather*}
By \cite[Theorem 5.15]{MM65}, it is enough to demonstrate that the supersymmetric algebra $\mathfrak{S}\left(\mathfrak{a}[x]^{\theta}\right)$ possesses no nonzero invariants under the $\mathbb{Z}_2$-graded adjoint action of $\mathfrak{a}_+\left[x^2\right]$.

Let $B\in \mathfrak{S}(\mathfrak{a}[x]^{\theta})$ be an $\mathfrak{a}_+[x^2]$-invariant element. Suppose $d$ is the maximal nonnegative integer such that
\begin{gather*}
B=\sum_{\lambda} B_{\lambda} (f_1^-x^d)^{\lambda_1}\cdots (f_q^-x^d)^{\lambda_q},
\end{gather*}
where the $q$-tuple $\lambda=(\lambda_1,\ldots,\lambda_q)\in\mathbb{Z}_+\times \cdots\times \mathbb{Z}_+$ and the coefficients $B_{\lambda}$ are the polynomials in $f_j^-x^r$ for $1\leqslant j\leqslant q$, $0\leqslant r<d$ and $f_i^+x^s$ for $1\leqslant i\leqslant p$, $s\geqslant 0$. Here each $\lambda_k$ for $1\leqslant k\leqslant q$ ranges over $\mathbb{Z}_+$ except that if some $\lambda_k>1$, with the $\mathbb{Z}_2$-grading $|f_k^-|=\bar{1}$, put $B_{\lambda}=0$.

The selection of $B$ leads to
\begin{gather*}
\left[f_i^+x,\,B\right]=0\quad \text{for}\quad 1\leqslant i\leqslant p.
\end{gather*}
For simplicity, set
\begin{gather*}
F_{\lambda}=(f_1^-x^d)^{\lambda_1}\cdots (f_q^-x^d)^{\lambda_q}.
\end{gather*}
If $|f_j^-|=\bar{1}$ for some $j$, we have
\begin{equation*}
\left[f_i^+x,\,(f_j^-x^d)^{\lambda_j}\right]=\begin{cases}
(-1)^{|f_i^+||f_j^-|}\sum_{k}f_k^-x^{d+1}, &\lambda_j=1, \\
0,  &\lambda_j=0.
\end{cases}
\end{equation*}
Then
\begin{align*}
\left[f_i^+x,\,B\right] =&\sum_{\lambda} \left[f_i^+x,\,B_{\lambda}\right]F_{\lambda}+\sum_{\lambda}(-1)^{|f_i^+||B_{\lambda}|}B_{\lambda}\left[f_i^+x,\,F_{\lambda}\right] \\
 =&\sum_{\lambda} \left[f_i^+x,\,B_{\lambda}\right]F_{\lambda}+\sum_{\lambda}(-1)^{|f_i^+||B_{\lambda}|}B_{\lambda}  \\
 &\times\sum_{j=1}^q (-1)^{|f_i^+|(|f_1^-|+\cdots+|f_j^-|)}\lambda_j (f_1^-x^d)^{\lambda_1}\cdots (f_j^-x^d)^{\lambda_q-1} \cdots (f_q^-x^d)^{\lambda_q}
   \sum_{k=1}^q c_{ij}^k f_k^-x^{d+1},
\end{align*}
which forces the coefficients of every $f_k^-x^{d+1}$ for $1\leqslant k\leqslant p$ to be zero since $\left[f_i^+x,\,B\right]=0$, that is to say,
\begin{gather*}
\sum_{\lambda}B_{\lambda}\sum_{j=1}^q (-1)^{|f_i^+|(|B_{\lambda}|+|f_1^-|+\cdots+|f_j^-|)}\lambda_jc_{ij}^k (f_1^-x^d)^{\lambda_1}\cdots (f_j^-x^d)^{\lambda_q-1} \cdots (f_q^-x^d)^{\lambda_q}=0,
\end{gather*}
for $1\leqslant i\leqslant p$, $1\leqslant k\leqslant q$. Hence, for any nonnegative integer $q$-tuple $\lambda'=(\lambda_1',\ldots,\lambda_q')$, it is established that
\begin{gather}\label{B:lamdba}
\sum_{j=1}^q (-1)^{|f_i^+|(|B_{\lambda'(j)}|+|f_1^-|+\cdots+|f_j^-|)}(\lambda_j'+1)c_{ij}^k B_{\lambda'(j)}=0,\quad 1\leqslant i\leqslant p,\ 1\leqslant k\leqslant q,
\end{gather}
where $\lambda'(j)$ denotes the $q$-tuple $(\lambda_1',\ldots,\lambda_j'+1,\ldots,\lambda_q')$.
Here we use the well-known formula:
\begin{gather*}
\left[f,\,a_1a_2\cdots a_m\right]=\sum_i (-1)^{|f|(|a_1|+\cdots+|a_{i-1}|)}a_1\cdots a_{i-1}\cdot\left[f,\,a_i\right]\cdot a_{i+1}\cdots a_m.
\end{gather*}

Fix $\lambda'$. Denote
\begin{gather*}
g_j^-=(-1)^{|f_i^+|(|B_{\lambda'(j)}|+|f_1^-|+\cdots+|f_j^-|)}(\lambda_j'+1)f_j^-,\quad 1\leqslant j\leqslant q.
\end{gather*}
The set of all elements $g_j^-$ also forms a basis of $\mathfrak{a}_-$. Pick $m_1,\ldots,m_{q}$ such that
\begin{gather}\label{B:lambda2}
\left[f_i^+,\,\sum_{j=1}^q m_j g_j^-\right]=0\quad\text{with}\quad 1\leqslant i\leqslant p.
\end{gather}
Since the centralizer of $\mathfrak{a}_+$ in $\mathfrak{a}$ is trivial, the system of linear equations \eqref{B:lambda2} involving the variables $m_1,\ldots,m_{q}$ has only zero solution. Additionally, \eqref{B:lambda2} is equivalent to the system of equations
\begin{gather*}
\sum_{j=1}^q (-1)^{|f_i^+|(|B_{\lambda'(j)}|+|f_1^-|+\cdots+|f_j^-|)}(\lambda_j'+1)c_{ij}^k b_j=0,\quad 1\leqslant i\leqslant p,\ 1\leqslant k\leqslant q,
\end{gather*}
Comparing with \eqref{B:lamdba}, we have $B_{\lambda'(j)}=0$.

Therefore, the equality $B_{\lambda}=0$ holds for all nonzero $\lambda$. Thus, the proof is complete.

\end{proof}

\vspace{1em}
Now, we consider a specific twisted polynomial current Lie superalgebra.
Let
\begin{gather*}
    \theta:\ E_{ij}\mapsto -(-1)^{|i||j|+|i|}\theta_i\theta_j E_{j'i'}
\end{gather*}
be the involution of the general linear Lie superalgebra $\mathfrak{gl}_{M|N}$ with order 2. The fixed sub-superalgebra under $\theta$ coincides with the ortho-symplectic Lie superalgebra $ \mathfrak{osp}_{M|N}$, i.e.,
\begin{gather*}
    \mathfrak{osp}_{M|N}:=\operatorname{span}\Big\{\,F_{ij}=E_{ij}-(-1)^{|i||j|+|i|}\theta_i\theta_j E_{j'i'}\,\Big|\,i,j\in I\,\Big\}.
\end{gather*}
Consider the twisted polynomial current Lie superalgebra $\mathfrak{gl}_{M|N}[x]^{\theta}$ with $\left(\mathfrak{gl}_{M|N}\right)_+=\mathfrak{osp}_{M|N}$.
Set
\begin{gather*}
    F_{ij}^{(r)}:=\left(E_{ij}+(-1)^r (-1)^{|i||j|+|i|}\theta_i\theta_j E_{j'i'}\right)x^{r-1}, \quad i,j\in I,\ r\in\mathbb{N}.
\end{gather*}
Note that the set of all elements $F_{ij}^{(r)}$ spans $\mathfrak{gl}_{M|N}[x]^{\theta}$.

Given a supersymmetric block matrix $\mathcal{G}$. The superalgebra $\mathrm{Y}^{tw}_{\mathcal{G}}\left(\mathfrak{osp}_{M|N}\right)$ is a sub-superalgebra of $\mathrm{Y}\left(\mathfrak{gl}_{M|N}\right)$. It inherits a filtration on $\mathrm{Y}^{tw}_{\mathcal{G}}\left(\mathfrak{osp}_{M|N}\right)$ such that
\begin{gather*}
    \mathcal{F}^{tw}_p=\mathcal{F}_p\cap \mathrm{Y}^{tw}_{\mathcal{G}}\left(\mathfrak{osp}_{M|N}\right),\quad \deg s_{ij}^{(r)}=r-1.
\end{gather*}
The graded algebra subject to this filtration is given by
\begin{gather*}
    \operatorname{gr}\mathrm{Y}^{tw}_{\mathcal{G}}\left(\mathfrak{osp}_{M|N}\right)
    :=\mathcal{F}^{tw}_0\oplus\left(\mathop{\bigoplus}\limits_{p=1}^{\infty}\mathcal{F}^{tw}_p/\mathcal{F}^{tw}_{p-1}\right).
\end{gather*}

\begin{proposition}\label{special}
The graded algebra $\operatorname{gr}\mathrm{Y}^{tw}_{\mathcal{G}}\left(\mathfrak{osp}_{M|N}\right)$ is isomorphic to the universal enveloping superalgebra corresponding to $\mathfrak{gl}_{M|N}[x]^{\theta}$ for any supersymmetric block matrix $\mathcal{G}$, that is,
$\operatorname{gr}\mathrm{Y}^{tw}_{\mathcal{G}}\left(\mathfrak{osp}_{M|N}\right)\cong \mathrm{U}\left(\mathfrak{gl}_{M|N}[x]^{\theta}\right)$.
\end{proposition}
\begin{proof}
By Proposition \ref{Ytw:eqv}, we only need to prove for $\mathcal{G}_0=\sum_k(-1)^{|k|}\theta_kE_{kk'}$. If there is no ambiguity, we omit the subscript $\mathcal{G}$.
From Proposition \ref{grU}, the mapping
\begin{gather*}
\overline{t}_{ij}^{(r)}\mapsto (-1)^{|i|}E_{ij}x^{r-1},\quad i,j\in I,\ r\in\mathbb{N},
\end{gather*}
forms an isomorphism between $\operatorname{gr}\mathrm{Y}\left(\mathfrak{gl}_{M|N}\right)$ and
$\mathrm{U}\left(\mathfrak{gl}_{M|N}[x]\right)$. 

Let $\bar{s}_{ij}^{(r)}$ denotes the image of $s_{ij}^{(r)}$ in the $r$-th component of $\operatorname{gr}\mathrm{Y}^{tw}\big(\mathfrak{osp}_{M|N}\big)$. The images of the elements $\overline{\mathfrak{s}}_{ij}^{(r)}$ under the following superalgebra homomorphism sequence
\begin{gather*}
\operatorname{gr}\mathrm{Y}^{tw}\left(\mathfrak{osp}_{M|N}\right)\hookrightarrow  \operatorname{gr}\mathrm{Y}\left(\mathfrak{gl}_{M|N}\right)  \rightarrow \mathrm{U}\left(\mathfrak{gl}_{M|N}[x]\right)
\end{gather*}
are equal to $-(-1)^{|j|}\theta_jF_{j'i}^{(r)}$,
which also spans $\mathfrak{gl}_{M|N}[x]^{\theta}$.

\end{proof}

\begin{corollary}
    The ortho-symplectic Lie superalgebra $\mathfrak{osp}_{M|N}$ is isomorphic to the superalgebra spanned by
    $$\sum_{k\in I}\left( E_{ik}g_{kj}+(-1)^{|i|}E_{jk}g_{ki}\right),\quad i,j\in I$$
    for any supersymmetric block matrix $\mathcal{G}=(g_{ij})$.
\end{corollary}

\subsection{Comments on the center $\mathcal{Z}\mathrm{Y}^{tw}\left(\mathfrak{osp}_{M|N}\right)$}
With the framework of the embedding map \eqref{embedding}, we regard $\mathrm{Y}^{tw}\big(\mathfrak{osp}_{M|N}\big)$ as a sub-superalgebra of $\mathrm{Y}\big(\mathfrak{gl}_{M|N}\big)$.
We present the twisted analogue of the central formulas as stated in Proposition \ref{centre:gl}. The following lemma will be utilized in Section \ref{se:quantumbereianian:1}.

\begin{lemma}\label{rho-Z}
The automorphism $\rho$ in \eqref{auto:Ygl4} acts on $z(u)$ given by
\begin{gather*}
\rho\left(z(u)\right)=z(-u-M+N)^{-1}.
\end{gather*}
\end{lemma}

\begin{proof}
Applying the automorphism \eqref{auto:Ygl5} to the both sides of \eqref{centre:2}, we get
\begin{gather*}
\delta_{ij}z(u)^{-1}=\sum_k (-1)^{|i||k|+|j||k|+|k|+|i|} \tilde{t}_{jk}(u+M-N)t_{ki}(u).
\end{gather*}
Replacing $u$ with $-u-M+N$, we obtain
\begin{gather}\label{central:2:1}
\delta_{ij}z(-u-M+N)^{-1}=\sum_k (-1)^{|i||k|+|j||k|+|k|+|i|} \tilde{t}_{jk}(-u)t_{ki}(-u-M+N).
\end{gather}
Then the lemma follows from the action of $\rho$ on \eqref{centre:1}.

\end{proof}


\begin{proposition}
In the twisted super Yangian $\mathrm{Y}^{tw}(\mathfrak{osp}_{M|N})$, we have
\begin{align}\label{centre:3}
&\sum_k (-1)^{|j||k|+|i||k|+|i|}\tilde{s}_{jk}(-u)s_{ki}(-u-M+N)=\delta_{ij}z(u)z(-u-M+N)^{-1}, \\ \label{centre:4}
&\sum_k (-1)^{|j||k|+|i||k|+|i|}s_{jk}(-u-M+N)\tilde{s}_{ki}(-u)=\delta_{ij}z(u)z(-u-M+N)^{-1}.
\end{align}
\end{proposition}
\begin{proof}
Let $\mathcal{G}=(g_{ij})$ be fixed as in \eqref{embedding}. Set
\begin{gather*}
    \mathcal{G}^{-1}=\sum_{ij}E_{ij}\otimes \tilde{g}_{ij}
\end{gather*}
such that $\sum_k \tilde{g}_{ik}g_{kj}=\delta_{ij}$.

We compute only equation \eqref{centre:3}, while equation \eqref{centre:4} can be verified using a similar approach. Using equation \eqref{central:2:1} and embedding \eqref{embedding}, the left side of \eqref{centre:3} is equal to
\begin{align*}
    &\quad\ \sum_{a,b}\sum_{c,d}\tilde{t}_{aj}(u)\tilde{g}_{ab}\bigg(\sum_k(-1)^{|k||b|+|k||c|+|k|+|c|}\tilde{t}_{bk}(-u)t_{kc}(-u-M+N)\bigg)g_{cd}t_{id}(u+M-N) \\
    &=\sum_{a,b,d}\tilde{t}_{aj}(u)\tilde{g}_{ab}g_{bd}t_{id}(u+M-N)z(-u-M+N)^{-1} \\
    &=\sum_{a}\tilde{t}_{aj}(u)t_{ia}(u+M-N)z(-u-M+N)^{-1}=\delta_{ij}z(u)z(-u-M+N)^{-1}.
\end{align*}

\end{proof}

Put $\mathfrak{z}(u)=z(u)z(-u-M+N)^{-1}$. It can be rewritten as a formal power series as follows
\begin{gather*}
\mathfrak{z}(u)=1+\mathfrak{z}^{(1)}u^{-1}+\mathfrak{z}^{(2)}u^{-2}+\mathfrak{z}^{(3)}u^{-3}+\mathfrak{z}^{(4)}u^{-4}+\cdots \ \in \ \mathrm{Y}^{tw}\big(\mathfrak{osp}_{M|N}\big)\big[\big[u^{-1}\big]\big].
\end{gather*}
Observe that $\mathfrak{z}^{(1)}=\mathfrak{z}^{(2)}=0$.
\begin{theorem}
The coefficients $\mathfrak{z}^{(3)}$, $\mathfrak{z}^{(4)}$, $\mathfrak{z}^{(5)}$, $\ldots$ are generators of the center of the twisted super Yangian $\mathrm{Y}^{tw}(\mathfrak{osp}_{M|N})$.
\end{theorem}
\begin{proof}
Combining with Proposition \ref{grU} and Corollary \ref{Na:Z:U}, the image of central elements $z^{(r)}$ for $r\geqslant 2$ in the universal enveloping superalgebra $\mathrm{U}\left(\mathfrak{gl}_{M|N}[x]\right)$ is equal to
\begin{gather*}
    -(r-1)\sum_{i\in I}E_{ii}x^{r-1}.
\end{gather*}
Then by Proposition \ref{special}, it follows that the image of the elements $\mathfrak{z}^{(2m+1)}$ for $m\geqslant 1$ in the universal enveloping superalgebra $\mathrm{U}\left(\mathfrak{gl}_{M|N}[x]^{\theta}\right)$ is given by
\begin{gather*}
y(m)\cdot x^{2m-1}\cdot E,
\end{gather*}
where $E$ denotes the identity matrix $E=\sum_{i\in I}E_{ii}$ and $y(m)$ is a finite sum related to $m$. Moreover, the image of the elements $\mathfrak{z}^{(2m+2)}$ for $m\geqslant 1$ in $\operatorname{gr}\mathrm{Y}^{tw}\left(\mathfrak{osp}_{M|N}\right)$ can be generated by $\mathfrak{z}^{(2m+1)}$. Hence, it suffices to show that the center of $\mathrm{U}\left(\mathfrak{gl}_{M|N}[x]^{\theta}\right)$ is generated by $x\cdot E$, $x^3\cdot E$, $x^5\cdot E$, $\ldots$
Let $\mathfrak{a}[x]^{\theta}$ be the quotient of $\mathfrak{gl}_{M|N}[x]^{\theta}$ by the ideal generated by $\mathbb{C}[x\cdot E,x^3\cdot E,x^5\cdot E,\ldots]$. Then the statement of this theorem follows from Lemma \ref{trivialc}.
\end{proof}

\subsection{Special twisted super Yangian}
 Regard $\mathrm{Y}^{tw}\left(\mathfrak{osp}_{M|N}\right)$ as a sub-superalgebra of $\mathrm{Y}\left(\mathfrak{gl}_{M|N}\right)$ under the embedding \eqref{embedding}. We can also define a ''special'' version of the twisted super Yangian as in $\mathrm{Y}\left(\mathfrak{gl}_{M|N}\right)$.
\begin{definition}
    The special twisted super Yangian $\mathrm{SY}^{tw}\left(\mathfrak{osp}_{M|N}\right)$ is the sub-superalgebra of $\mathrm{Y}^{tw}\left(\mathfrak{osp}_{M|N}\right)$ determined by
    \begin{gather*}
    \mathrm{SY}^{tw}\left(\mathfrak{osp}_{M|N}\right)
               =\mathrm{Y}\left(\mathfrak{sl}_{M|N}\right)\cap \mathrm{Y}^{tw}\left(\mathfrak{osp}_{M|N}\right).
    \end{gather*}
\end{definition}

Since the map
\begin{gather}\label{Ytw:auto:1}
    \mu_f^{tw}:\ S(u)\mapsto f(u)S(u)
\end{gather}
also defines an automorphism of $\mathrm{Y}^{tw}\left(\mathfrak{osp}_{M|N}\right)$ for all $f(u)\in 1+u^{-1}\mathbb{C}\left[\left[u^{-1}\right]\right]$, the special twisted super Yangian has an equivalent definition.
\begin{definition}
    The special twisted super Yangian $\mathrm{SY}^{tw}\left(\mathfrak{osp}_{M|N}\right)$ is the sub-superalgebra of $\mathrm{Y}^{tw}\left(\mathfrak{osp}_{M|N}\right)$, which consists of elements preserved by all automorphisms \eqref{Ytw:auto:1}.  That is,
    $$\mathrm{SY}^{tw}(\mathfrak{osp}_{M|N}):=\left\{y\in\mathrm{Y}^{tw}(\mathfrak{osp}_{M|N})|\ \mu_f^{tw}(y)=y \text{ for all }f\right\}.$$
\end{definition}

\begin{proposition}
    For $M\neq N$, the superalgebra $\mathrm{Y}^{tw}\left(\mathfrak{osp}_{M|N}\right)$ has the following decomposition
    \begin{gather}\label{tw:decomposition}
        \mathrm{Y}^{tw}\left(\mathfrak{osp}_{M|N}\right)
                =\mathcal{Z}\mathrm{Y}^{tw}\left(\mathfrak{osp}_{M|N}\right)\otimes \mathrm{SY}^{tw}\left(\mathfrak{osp}_{M|N}\right).
    \end{gather}
\end{proposition}

\begin{proof}
    This proof is very similar to \cite[Theorem 2.9.2]{Mo07} (or see \cite[Proposition 2.16]{MNO96}). Since the second filtration mentioned in Section \ref{se:Superyangian:1} preserves the grading of generators, all coefficients of $z(u)$ are even elements. Therefore, the center $\mathcal{Z}\mathrm{Y}\left(\mathfrak{gl}_{M|N}\right)$ is an associative commutative algebra.

   Let $\mathcal{A}$ be an arbitrary associative commutative algebra. Due to \cite[Proposition 2.15]{MNO96},
    for any series $a(u)\in 1+u^{-1}\mathcal{A}\left[\left[u^{-1}\right]\right]$, there is a unique series $\mathring{a}(u)\in 1+u^{-1}\mathcal{A}\left[\left[u^{-1}\right]\right]$ such that
    \begin{gather*}
        a(u)=\mathring{a}(u)\mathring{a}(u-1)\cdots \mathring{a}(u-K+1)
    \end{gather*}
    for any $K\in\mathbb{N}$.

    If $M>N$, the statement above immediately implies that there is a unique series
    $$\mathring{B}(u)\in 1+u^{-1}\mathcal{Z}\mathrm{Y}\big(\mathfrak{gl}_{M|N}\big)\left[\left[u^{-1}\right]\right]$$
    such that
    \begin{gather*}
        \mathfrak{B}_{M|N}(u)=\mathring{B}(u+M-N-1)\cdots \mathring{B}(u).
    \end{gather*}
    By \eqref{qLiou-gl}, we have
    \begin{gather*}
        z(u)=\mathring{B}(u+M-N)\mathring{B}(u)^{-1}.
    \end{gather*}
    Following from \eqref{centre:1}, the automorphism $\mu_f$ in \eqref{auto:Ygl1} maps $\mathring{B}(u)$ to $f(u)\mathring{B}(u)$.
Thus, the coefficients of
\begin{gather*}
    \mathring{t}_{ij}(u)=\mathring{B}(u)^{-1}t_{ij}(u)\in \mathrm{Y}\left(\mathfrak{gl}_{M|N}\right)\big[\big[u^{-1}\big]\big].
\end{gather*}
are invariant under the automorphism $\mu_f$, which forces $\mathring{t}_{ij}(u)\in \mathrm{Y}(\mathfrak{sl}_{M|N})\big[\big[u^{-1}\big]\big]$.

 Set
\begin{gather*}
    \mathring{s}_{ij}(u)=\mathring{B}^{-1}(u)\mathring{B}^{-1}(-u)s_{ij}(u)\in \mathrm{Y}^{tw}(\mathfrak{osp}_{M|N})\big[\big[u^{-1}\big]\big].
\end{gather*}
This implies
\begin{align*}
    \mathring{s}_{ij}(u)=\sum_{k,l\in I}(-1)^{|i||j|+|i||k|+|j|+|k|}g_{kl}\mathring{t}_{ik}(u)\mathring{t}_{jl}(-u),
\end{align*}
by Proposition \ref{prop:embed}. Since
    \begin{align*}
        \mathfrak{z}(u)&=z(u)z(-u-M+N)^{-1} \\
        &=\mathring{B}^{-1}(u)\mathring{B}^{-1}(-u)\mathring{B}(u+M-N)\mathring{B}(-u-M+N),
    \end{align*}
    all coefficients of $\mathring{B}(u)\mathring{B}(-u)$ are included in the center $\mathcal{Z}\mathrm{Y}^{tw}\big(\mathfrak{osp}_{M|N}\big)$.
    The coefficients of $\mathring{s}_{ij}(u)$ are also invariant under any automorphisms $\mu_g^{tw}$ for $g(u)=f(u)f(-u)$ in \eqref{Ytw:auto:1}.
    Hence, we have the following decomposition
\begin{gather}\label{tw:decomposition2}
\mathrm{Y}^{tw}\left(\mathfrak{osp}_{M|N}\right)=\mathcal{Z}\mathrm{Y}^{tw}\left(\mathfrak{osp}_{M|N}\right)\mathrm{SY}^{tw}\left(\mathfrak{osp}_{M|N}\right).
    \end{gather}

 In the case of $N<M$, we can take
 $$\mathring{B}'(u)\in 1+u^{-1}\mathcal{Z}\mathrm{Y}\big(\mathfrak{gl}_{M|N}\big)\left[\left[u^{-1}\right]\right]$$  such that
 \begin{gather*}
     \mathfrak{B}_{M|N}(u)=\mathring{B}'(-u+1)\cdots \mathring{B}'(-u+N-M).
 \end{gather*}
 Select $\mathring{t}'_{ij}(u)=\mathring{B}'(-u)t_{ij}(u)$ and $\mathring{s}'_{ij}(u)=\mathring{B}'(u)\mathring{B}'(-u)s_{ij}(u)$ to replace $\mathring{t}_{ij}(u)$ and $\mathring{s}_{ij}(u)$, respectively. Then,
we also have \eqref{tw:decomposition2}.

Moreover, the inclusions
\begin{gather*}
    \mathcal{Z}\mathrm{Y}^{tw}\left(\mathfrak{osp}_{M|N}\right)\subset
              \mathcal{Z}\mathrm{Y}\left(\mathfrak{gl}_{M|N}\right),\quad
      \mathrm{SY}^{tw}\left(\mathfrak{osp}_{M|N}\right)\subset
             \mathrm{Y}\left(\mathfrak{sl}_{M|N}\right)
\end{gather*}
lead to \eqref{tw:decomposition}.

\end{proof}
It is clear that the center $\mathcal{Z}\mathrm{Y}^{tw}\left(\mathfrak{osp}_{M|N}\right)$ is an associative commutative algebra. Moreover, we also have the following

\begin{corollary}
    For $M\neq N$, the superalgebra $\mathrm{SY}^{tw}\left(\mathfrak{osp}_{M|N}\right)$ is isomorphic to the quotient of $\mathrm{Y}^{tw}\left(\mathfrak{osp}_{M|N}\right)$ by the ideal generated by $\mathfrak{z}^{(3)}$, $\mathfrak{z}^{(4)}$, $\mathfrak{z}^{(5)}$, $\ldots$.
\end{corollary}

\begin{corollary}
    For $M\neq N$, the superalgebra $\mathrm{SY}^{tw}\left(\mathfrak{osp}_{M|N}\right)$ can be regarded as a left coideal sub-superalgebra of $\mathrm{Y}\left(\mathfrak{sl}_{M|N}\right)$ under the comultiplication $\Delta$.
\end{corollary}

\section{Quantum Berezinian for twisted super Yangian}\label{se:quantumbereianian}
In this section, we will present the explicit form of the quantum Berezinian for the generators of $\mathrm{Y}^{tw}\left(\mathfrak{osp}_{M|N}\right)$, referring to the references \cite{Mo95,Mo07,MNO96,Na20}.

\subsection{Definition of quantum Berezinian in twisted type}\label{se:quantumbereianian:1}

Define a homomorphism $\pi_m$ from the symmetric group $\mathfrak{S}_m$ to the $m$-tensor product algebra $\left(\operatorname{End}V\right)^{\otimes m}$ such that
\begin{gather*}
\sigma_{ij}=(ij)\mapsto P_{i,j},\quad \text{for }1\leqslant i<j\leqslant m.
\end{gather*}

Denote by $G^{(m)}$ and $H^{(m)}$ the respective images of the anti-symmetrizer and symmetrizer:
\begin{gather*}
\sum_{\sigma\in\mathfrak{S}_m}\operatorname{sgn}(\sigma)\cdot\sigma\qquad \text{and}\qquad \sum_{\sigma\in\mathfrak{S}_m}\sigma
\end{gather*}
under the action of $\pi_m$. Observe that for $m\geqslant 2$,
\begin{align*}
&G^{(m)}=\left(1-P_{1,m}-\cdots-P_{1,2}\right)\cdots\left(1-P_{m-2,m}-P_{m-2,m-1}\right)
    \left(1-P_{m-1,m}\right), \\
&H^{(m)}=\left(1+P_{1,m}+\cdots+P_{1,2}\right)\cdots\left(1+P_{m-2,m}+P_{m-2,m-1}\right)
    \left(1+P_{m-1,m}\right).
\end{align*}
Set
\begin{gather*}
    \sigma_i=\sigma_{i,i+1},\quad \epsilon_i=(-1)^{|i||i+1|}.
\end{gather*}
We define the function $\epsilon$ on $\mathfrak{S}_m$ by
\begin{gather*}
    \epsilon(\sigma)=\epsilon_{i_1}\epsilon_{i_2}\cdots \epsilon_{i_k}\quad\text{for}\quad \sigma=\sigma_{i_1}\sigma_{i_2}\cdots \sigma_{i_k}\in\mathfrak{S}_m.
\end{gather*}
It is easy to see that $\epsilon(\sigma)=\operatorname{sgn}\sigma$ for each permutation $\sigma\in\mathfrak{S}$.

\begin{lemma}
In $V^{\otimes(M+N)}$, we have
\begin{align*}
&G^{(m)}(e_{i_1}\otimes \cdots \otimes e_{i_m})=\sum_{\sigma\in \mathfrak{S}_m}\epsilon(\sigma)\operatorname{sgn}(\sigma)(e_{i_{\sigma(1)}}\otimes \cdots \otimes e_{i_{\sigma(m)}}), \\
&H^{(m)}(e_{i_1}\otimes \cdots \otimes e_{i_m})=\sum_{\sigma\in \mathfrak{S}_m}\epsilon(\sigma)(e_{i_{\sigma(1)}}\otimes \cdots \otimes e_{i_{\sigma(m)}}).
\end{align*}

\end{lemma}

\begin{proof}
    This lemma can be proved by induction on $m$.

\end{proof}

\begin{corollary}\label{GH-del}
If $\{i_1,\ldots,i_m\}\subseteq\{1,\ldots,M\}$, then
\begin{gather*}
G^{(m)}(e_{i_1}\otimes \cdots \otimes e_{i_m})=\sum_{\sigma\in \mathfrak{S}_m}\operatorname{sgn}(\sigma)\cdot(e_{i_{\sigma(1)}}\otimes \cdots \otimes e_{i_{\sigma(m)}}).
\end{gather*}
If $\{i_1,\ldots,i_m\}\subseteq\{M+1,\ldots,M+N\}$, then
\begin{gather*}
H^{(m)}(e_{i_1}\otimes \cdots \otimes e_{i_m})=\sum_{\sigma\in \mathfrak{S}_m}\operatorname{sgn}(\sigma) \cdot (e_{i_{\sigma(1)}}\otimes \cdots \otimes e_{i_{\sigma(m)}}).
\end{gather*}
\end{corollary}

\begin{corollary}\label{GH-del2}
Given a permutation $\sigma\in\mathfrak{S}_m$. If $\{i_1,\ldots,i_m\}\subseteq\{1,\ldots,M\}$, then
\begin{gather*}
G^{(m)}(e_{i_{\sigma(1)}}\otimes \cdots \otimes e_{i_{\sigma(m)}})=\operatorname{sgn}(\sigma) \cdot G^{(m)}(e_{i_1}\otimes \cdots \otimes e_{i_m}).
\end{gather*}
If $\{i_1,\ldots,i_m\}\subseteq\{M+1,\ldots,M+N\}$, then
\begin{gather*}
H^{(m)}(e_{i_{\sigma(1)}}\otimes \cdots \otimes e_{i_{\sigma(m)}})=\operatorname{sgn}(\sigma)\cdot H^{(m)}(e_{i_1}\otimes \cdots \otimes e_{i_m}).
\end{gather*}
\end{corollary}

Consider the following rational function in $\left(\operatorname{End}V\right)^{\otimes m}\left[\left[u_1^{-1},\ldots,u_m^{-1}\right]\right]$ for $m\geqslant 2$:
\begin{gather*}
R(u_1,u_2,\ldots,u_m)=R_{m-1,m}\left(R_{m-2,m}R_{m-2,m-1}\right)\cdots
     \left(R_{1,m}\cdots R_{1,2}\right),
\end{gather*}
where $R_{ij}$ for $i<j$ denotes $R_{ij}(u_i-u_j)$ if no confusion.

The next lemma follows from the well-known ``multiplicative formula'' for the anti-symmetrizer and symmetrizer in $\mathbb{C}[\mathfrak{S}_m]$:
\begin{gather*}
    \sum_{\sigma\in\mathfrak{S}_m}\operatorname{sgn}\sigma \cdot \sigma =\mathop{\overleftarrow{\prod}}\limits_{i=1}^{m-1}
        \mathop{\overleftarrow{\prod}}\limits_{j=i+1}^{m} \left(1-\frac{\sigma_{ij}}{j-i}\right),\quad
    \sum_{\sigma\in\mathfrak{S}_m}\sigma =\mathop{\overleftarrow{\prod}}\limits_{i=1}^{m-1}
        \mathop{\overleftarrow{\prod}}\limits_{j=i+1}^{m} \left(1+\frac{\sigma_{ij}}{j-i}\right).
\end{gather*}
\begin{lemma}\label{GH-ex}
(1)\ Choose $u_1,\ldots,u_m$ such that $u_i-u_{i+1}=1$ for $i\in\{1,2,\ldots,m-1\}$, then we have
\begin{gather*}
G^{(m)}=R(u_1,\ldots,u_m).
\end{gather*}

(2)\ Choose $u_{1},\ldots,u_{m}$ such that $u_i-u_{i+1}=-1$ for $i\in\{1,2,\ldots,m-1\}$, then we have
\begin{gather*}
H^{(m)}=R(u_{1},\ldots,u_{m}).
\end{gather*}
\end{lemma}

\begin{proof}
    This lemma can be proven directly by the action of $\pi_m$.

\end{proof}

\begin{lemma}\label{GHT:eqs}
In $\left(\operatorname{End}V\right)^{\otimes m}\otimes \mathrm{Y}(\mathfrak{gl}_{M|N})\left[\left[u^{-1}\right]\right]$,
\begin{align}\label{GH:eq1}
G^{(m)}T_1(u)\cdots T_{m}(u-m+1)&=T_m(u-m+1)\cdots T_1(u)G^{(m)}, \\ \label{GH:eq2}
H^{(m)}T_1^{\ast}(u)\cdots T_m^{\ast}(u+m-1)&=T_m^{\ast}(u+m-1)\cdots T_1^{\ast}(u)H^{(m)}.
\end{align}
where we denote $G^{(m)}=G^{(m)}\otimes 1$ and $H^{(m)}=H^{(m)}\otimes 1$ to simplify matters.
\end{lemma}

\begin{proof}
    We check the relation \eqref{GH:eq1} first.
    By Lemma \ref{GH-ex}, it suffices to show that
    \begin{gather*}
        \left(R(u_1,\ldots,u_m)\otimes 1\right)T_1(u_1)\cdots T_{m}(u_m)
           =T_m(u_m)\cdots T_1(u_1)\left(R(u_1,\ldots,u_m)\otimes 1\right)
    \end{gather*}
    for $u_i=u-i+1$. It can be proved by induction on $i\in\{1,\ldots,m-1\}$ owing to
    \begin{gather*}
        \left(R_{i,m}\cdots R_{i,i+1}\otimes 1\right)T_i(u_i)T_{i+1}(u_{i+1})\cdots T_m(u_m)=T_{i+1}(u_{i+1})\cdots T_m(u_m)T_i(u_i)\left(R_{i,m}\cdots R_{i,i+1}\otimes 1\right).
    \end{gather*}

    Since the mapping \eqref{auto:Ygl5} is an automorphism of $\mathrm{Y}\big(\mathfrak{gl}_{M|N}\big)$, we obtain
    \begin{gather*}
        R_{12}(u-v)T_1^{\ast}(u)T_2^{\ast}(v)=T_2^{\ast}(v)T_1^{\ast}(u)R_{12}(u-v),
    \end{gather*}
    which allows us to prove \eqref{GH:eq2} as similarly as the first relation.
\end{proof}

Let $\mathcal{I}\in\operatorname{End}V$ and $\mathcal{J}\in\operatorname{End}V$ be the respective projections of the superspace $V$ onto its even subspace $\mathbb{C}^{M|0}$ and odd subspace $\mathbb{C}^{0|N}$. This means that
\begin{gather*}
    \mathcal{I}(e_i)=\delta_{0,|i|}\,e_i,\quad \mathcal{J}(e_i)=\delta_{1,|i|}\,e_i.
\end{gather*}
Denote $\Im=\mathcal{I}_1\cdots \mathcal{I}_M\mathcal{J}_{M+1}\cdots\mathcal{J}_{M+N}$. The action of $\Im$ on the superspace $$V^{\otimes (M+N)}\setminus \left(C^{M|0}\right)^{\otimes M}\otimes \left(C^{0|N}\right)^{\otimes N}$$ must be zero.
Since the identities
    \begin{gather*}
        \mathcal{I}_iP_{i,j}=P_{i,j}\mathcal{I}_j, \quad \mathcal{I}_jP_{i,j}=P_{i,j}\mathcal{I}_i,
    \end{gather*}
    we have in $u^{-1}$
    \begin{gather}\label{IRcomm}
        \mathcal{I}_i\mathcal{I}_jR_{i,j}(u)=R_{i,j}(u)\mathcal{I}_i\mathcal{I}_j.
    \end{gather}
    Similarly,
    \begin{gather}\label{JRcomm}
        \mathcal{J}_i\mathcal{J}_jR_{i,j}(u)=R_{i,j}(u)\mathcal{J}_i\mathcal{J}_j.
    \end{gather}
    These two equalities \eqref{IRcomm} and \eqref{JRcomm} lead to $\Im A^{(M|N)}=A^{(M|N)}\Im$. Combining with Lemma \ref{GHT:eqs}, we have
\begin{align*}
    &A^{(M|N)}\Im  T_1(w_1)\cdots T_{M}(w_M) T^{\ast}_{M+1}(w_{M+1})\cdots T^{\ast}_{M+N}(w_{M+N})\Im \\
    =&\Im T^{\ast}_{M+N}(w_{M+N})\cdots T^{\ast}_{M+1}(w_{M+1})T_{M}(w_M)\cdots T_1(w_1) A^{(M|N)} \Im,
\end{align*}
for $A^{(M|N)}=G^{(M)}\otimes H^{(N)}\otimes 1$, and
\begin{equation*}
    w_i=\begin{cases}
        u+M-N-i, &\text{if}~~i\in I_+, \\
        u-M-N-1+i, &\text{if}~~i\in I_-.
    \end{cases}
\end{equation*}
The operator $\Im$ and Corollary \ref{GH-del2} ensure that the image of $A^{(M|N)}\Im $ on $V^{\otimes (M+N)}$ is one-dimensional. Then there exists a formal power series
    \begin{gather*}
        \mathfrak{D}_{M|N}(u)\in 1+u^{-1}\mathrm{Y}(\mathfrak{gl}_{M|N})\left[\left[u^{-1}\right]\right]
    \end{gather*}
    such that in $\operatorname{End} V^{\otimes (M+N)}\otimes \mathrm{Y}(\mathfrak{gl}_{M|N})\left[\left[u^{-1}\right]\right]$,
    \begin{gather}\label{GH-QB:eq1}
\Im A^{(M|N)}\mathfrak{D}_{M|N}(u)\Im=\Im A^{(M|N)} T_1(w_1)\cdots T_{M}(w_M) T^{\ast}_{M+1}(w_{M+1})\cdots T^{\ast}_{M+N}(w_{M+N})\Im.
\end{gather}

\begin{theorem}\label{GH-QB}
The following equation hold in $\mathrm{Y}(\mathfrak{gl}_{M|N})\left[\left[ u^{-1}\right]\right]$,
    \begin{gather}\label{GH-QB:eqbd}
        \mathfrak{D}_{M|N}(u)=\mathfrak{B}_{M|N}(u).
    \end{gather}
    That is,
    \begin{gather}\label{GH-QB:eq2}
\Im A^{(M|N)}\mathfrak{B}_{M|N}(u)\Im=\Im A^{(M|N)} T_1(w_1)\cdots T_{M}(w_M) T^{\ast}_{M+1}(w_{M+1})\cdots T^{\ast}_{M+N}(w_{M+N})\Im.
\end{gather}
\end{theorem}

\begin{proof}
    By Corollary \ref{GH-del}, $G^{(M)}$ acts non-trivially on the vector $e_{i_1}\otimes\ldots\otimes e_{i_M}\in V^{\otimes M}$ only if the indices $i_1,\ldots,i_M\in\{1,\ldots,M\}$ are pairwise distinct. In other words, there exists some permutation $\sigma\in\mathfrak{S}_M$ such that $i_k=\sigma(k)$ for each $k$. So does $H^{(N)}$.

    Applying the right side of \eqref{GH-QB:eq1} to both sides of the vector $v=e_{\eta(1)}\otimes\cdots \otimes e_{\eta(M)} \otimes e_{M+\eta'(1)}\otimes \cdots\otimes  e_{M+\eta'(N)}$ for $\eta\in\mathfrak{S}_M$ and $\eta'\in \mathfrak{S}_N$, we obtain
    \begin{align*}
       &A^{(M|N)}(e_1\otimes \cdots\otimes e_{M+N})\sum_{\sigma\in\mathfrak{S}_M}\operatorname{sgn}(\sigma)\cdot t_{\sigma(1),\eta(1)}(w_1)\cdots t_{\sigma(M),\eta(M)}(w_M)  \\
       &\qquad\qquad\qquad\qquad\qquad\qquad\times\sum_{\sigma\in\mathfrak{S}_N}
         \operatorname{sgn}(\sigma) \cdot t^{\ast}_{M+\sigma(1),M+\eta'(1)}(w_{\overline{1}})\cdots t^{\ast}_{M+\sigma(N),M+\eta'(N)}(w_{\overline{N}}).
    \end{align*}
    The left side of \eqref{GH-QB:eq1} acts on $v$ to yield
    \begin{gather*}
        \operatorname{sgn}(\eta)\operatorname{sgn}(\eta')A^{(M|N)}(e_1\otimes \cdots\otimes e_{M|N}),
    \end{gather*}
    which implies that
    \begin{align*}
       &\mathfrak{D}_{M|N}(u)= \sum_{\sigma\in\mathfrak{S}_M}\operatorname{sgn}(\sigma\eta)\cdot t_{\sigma(1),\eta(1)}(w_1)\cdots t_{\sigma(M),\eta(M)}(w_M)  \\
       &\qquad\qquad\qquad\times\sum_{\sigma\in\mathfrak{S}_N}
         \operatorname{sgn}(\sigma\eta')\cdot t^{\ast}_{M+\sigma(1),M+\eta'(1)}(w_{\overline{1}})\cdots t^{\ast}_{M+\sigma(N),M+\eta'(N)}(w_{\overline{N}}).
    \end{align*}
    Hence, by Corollary \ref{GH-del2} and the definition of $\mathfrak{B}_{M|N}(u)$, we get \eqref{GH-QB:eqbd}.
    
\end{proof}

There exists an alternative realization to define the quantum Berezinian in terms of generators $T(u)$:
\begin{proposition}
The following equations hold in $\mathrm{Y}(\mathfrak{gl}_{M|N})\left[\left[ u^{-1}\right]\right]$,
    \begin{align*}
        \mathfrak{B}_{M|N}(u) &=\operatorname{Str}_{1,\cdots,M+N}\,\Im A^{(M|N)}T_1(w_1)\cdots T_{M}(w_M) T^{\ast}_{M+1}(w_{M+1})\cdots T^{\ast}_{M+N}(w_{M+N}) \\
        &=\operatorname{Str}_{1,\cdots,M+N}\,\Im A^{(M|N)} T_1(w_1)\cdots T_{M}(w_M) T^{\ast}_{M+1}(w_{M+1})\cdots T^{\ast}_{M+N}(w_{M+N})\Im.
    \end{align*}
\end{proposition}

\begin{proof}
We only need to establish the validity of the first equation. Subsequently, the second one can be derived straightforwardly. Due to $A^{(M|N)}\Im=\Im A^{(M|N)}$, one has
    \begin{align*}
        &\Im A^{(M|N)}T_1(w_1)\cdots T_{M}(w_M) T^{\ast}_{M+1}(w_{M+1})\cdots T^{\ast}_{M+N}(w_{M+N}) \\
        =&
           \sum_{k_1,\ldots,k_M\in I}E_{1,k_1}\otimes \cdots \otimes E_{M,k_M}\otimes 1^{\otimes N}\otimes \sum_{\sigma\in\mathfrak{S}_M}\operatorname{sgn}(\sigma)t_{\sigma(1),k_1}(w_1)\cdots t_{\sigma(M),k_M}(w_M)  \\
       &\times \sum_{h_1,\ldots,h_N\in I} 1^{\otimes M}\otimes E_{M+1,h_1}\otimes \cdots\otimes E_{M+N,h_N}\otimes \sum_{\sigma\in\mathfrak{S}_N}
         \operatorname{sgn}(\sigma) t^{\ast}_{M+\sigma(1),h_1}(w_{\overline{1}})\cdots t^{\ast}_{M+\sigma(N),h_N}(w_{\overline{N}}).
    \end{align*}
    The action of $\operatorname{Str}_{1,\cdots,M+N}$ on the right side of the equation above immediately gives the expression of $\mathfrak{B}_{M|N}(u)$.

\end{proof}

Let $\mathfrak{P}_{M|N}(u)$ be the series obtained from the quantum Berezinian $\mathfrak{B}_{M|N}(u)$ under the action of $\rho$, that is,
$$\mathfrak{P}_{M|N}(u):=\rho\left(\mathfrak{B}_{M|N}(u)\right).$$
Then, by Lemma \ref{rho-Z}, we have
\begin{gather*}
\mathfrak{z}(u)=z(u)z(-u-M+2n)^{-1}=\frac{\mathfrak{B}_{M|N}(u+1)\mathfrak{P}_{M|N}(u+1)}{\mathfrak{B}_{M|N}(u)\mathfrak{P}_{M|N}(u)}.
\end{gather*}

\begin{remark}
In $\mathrm{Y}(\mathfrak{gl}_N)$, it is known that $\rho\left(\mathrm{qdet}T(u)\right)=\mathrm{qdet}T(-u+N-1)$. However, $$\mathfrak{P}_{M|N}(u)=\rho\Big(\mathfrak{B}_{M|N}(u)\Big)\neq \mathfrak{B}_{M|N}(-u+N+1),$$
except when $M=N$ or $M=0$. For instance,
\begin{gather*}
\mathfrak{B}_{1|2}(u)=t_{11}(u-2)\big(\tilde{t}_{22}(u-2)\tilde{t}_{33}(u-1)-\tilde{t}_{23}(u-2)\tilde{t}_{32}(u-1)\big).
\end{gather*}
In this case, by the defining relation \eqref{Y1} and the automorphism \eqref{auto:Ygl5}, we know
\begin{align*}
\mathfrak{P}_{1|2}(u)&=t_{11}(-u+2)\big(\tilde{t}_{22}(-u+2)\tilde{t}_{33}(-u+1)-\tilde{t}_{32}(-u+2)\tilde{t}_{23}(-u+1)\big) \\
&=t_{11}(-u+2)\big(\tilde{t}_{22}(-u+1)\tilde{t}_{33}(-u+2)-\tilde{t}_{23}(-u+1)\tilde{t}_{32}(-u+2)\big) \\
&\neq \mathfrak{B}_{1|2}(-u+3).
\end{align*}
\end{remark}

\begin{corollary}
  Let $w_i$ be the same parameters as in Theorem \ref{GH-QB}, for $1\leqslant i \leqslant M+N$. Then the following equation holds in  $\operatorname{End}V^{\otimes (M+N)}\otimes \mathrm{Y}\left(\mathfrak{gl}_{M|N}\right)\left[\left[u^{-1}\right]\right]$,
\begin{gather}\label{GH-QB:eq3}
\Im A^{(M|N)}\mathfrak{P}_{M|N}(u)\Im
=\Im A^{(M|N)} T_1^t(-w_1)\cdots  T_{M}^t(-w_M) \widetilde{T}_{M+1}(-w_{M+1})\cdots \widetilde{T}_{M+N}(-w_{M+N})\Im.
\end{gather}
\end{corollary}
\begin{corollary}
Let $w_i$ be the same parameters as in Theorem \ref{GH-QB}, for $1\leqslant i \leqslant M+N$. Then the following equations hold in $\operatorname{End}V^{\otimes (M+N)}\otimes \mathrm{Y}\left(\mathfrak{gl}_{M|N}\right)\left[\left[u^{-1}\right]\right]$,
\begin{align}\label{GH-QB:eq4}
\Im A^{(M|N)}\mathfrak{B}_{M|N}(u) \Im
&=\Im A^{(M|N)} T_1(w_1)\cdots T_{M}(w_M) \widetilde{T}^{\iota}_{M+1}(w_{M+1})\cdots \widetilde{T}^{\iota}_{M+N}(w_{M+N})\Im, \\ \label{GH-QB:eq5}
\Im A^{(M|N)}\mathfrak{P}_{M|N}(u) \Im
&=\Im A^{(M|N)} T_1^{\iota}(-w_1)\cdots  T_{M}^{\iota}(-w_M) \widetilde{T}_{M+1}(-w_{M+1})\cdots \widetilde{T}_{M+N}(-w_{M+N})\Im.
\end{align}
\end{corollary}
\begin{proof}
The equation \eqref{GH-QB:eq4} is follows from multiplying both sides of the equations \eqref{GH-QB:eq2} by $$\mathcal{G}_{M+1}\mathcal{G}_{M+2}\cdots\mathcal{G}_{M+N}, \text{and }\mathcal{G}_{M+N}^{-1}\cdots\mathcal{G}_{M+2}^{-1}\mathcal{G}_{M+1}^{-1}$$
on the left and right, respectively. Due to
$\mathcal{G}_i\mathcal{G}_jP_{ij}\mathcal{G}_j^{-1}\mathcal{G}_i^{-1}=P_{ij}$ for $i<j$, one has
\begin{gather*}
\mathcal{G}_{M+1}\cdots\mathcal{G}_{M+N}\left(1^{\otimes M}\otimes H^{(N)}\right)\mathcal{G}_{M+N}^{-1}\cdots\mathcal{G}_{M+1}^{-1}=1^{\otimes M}\otimes H^{(N)}.
\end{gather*}
Thus, the equation \eqref{GH-QB:eq4} holds since $\mathcal{G}_i$ commutes with $\mathcal{I}_j$ (resp. $\mathcal{J}_j$, $T_k$, $\widetilde{T}_k^t$) for every $i,j\in I$ and $k\neq i$.
Equation \eqref{GH-QB:eq5} can be checked similarly by \eqref{GH-QB:eq3}.
\end{proof}

Now, to optimize processing, we will address the computation of the quantum Berezinian for $\mathcal{Y}^{tw}\left(\mathfrak{osp}_{M|N}\right)$ associated with $\mathcal{G}=(g_{ij})$. By Proposition \ref{Ytw:eqv}, we choose $\mathcal{G}=\mathcal{G}_0$ as in Remark \ref{BRtw}. In this case, 
$$P^{\iota_1}=P^{\iota_2},\quad \text{and}\quad (P^{\iota_1})^{\iota_1}=P.$$
According to the isomorphism $\phi$, we identify with
\begin{gather*}
    s_{ij}(u)=\theta_j\mathfrak{s}_{ij'}(u),\quad \tilde{s}_{ij}(u)=\theta_i\tilde{\mathfrak{s}}_{i'j}(u).
\end{gather*}
This implies that
\begin{gather*}
    \mathfrak{z}(u)=\sum_k (-1)^{|j||k|+|i||k|+|j|+|k|}\tilde{\mathfrak{s}}_{ik}(-u)\mathfrak{s}_{ki}(-u-M+N).
\end{gather*}

Set $\mathcal{S}_i=\mathcal{S}_i(u_i)$, $\widetilde{\mathcal{S}}_i=\widetilde{\mathcal{S}}_i\left(-u_i-\frac{M-N}{2}\right)$, and  $R_{ij}^{\iota}=R_{ji}^{\iota}=R_{i,j}^{\iota_i}(-u_i-u_j)$ for $i<j$. Denote
\begin{gather*}
\langle \mathcal{S}_{k_1},\ldots,\mathcal{S}_{k_m} \rangle:=\mathcal{S}_{k_1}(R_{k_1k_2}^{\iota}\cdots R_{k_1k_m}^{\iota})\mathcal{S}_{k_2}(R_{k_2k_3}^{\iota}\cdots R_{k_2k_m}^{\iota})\cdots \mathcal{S}_{k_m}.
\end{gather*}

\begin{lemma}\label{GH-RS}
    In $\operatorname{End}V^{\otimes m}\otimes \mathcal{Y}^{tw}(\mathfrak{osp}_{M|N})\left[\left[u^{-1}\right]\right]$, we have
    \begin{align*}
        (G^{(m)}\otimes 1)\langle \mathcal{S}_1,\ldots,\mathcal{S}_m \rangle&=
          \langle \mathcal{S}_m,\ldots,\mathcal{S}_1 \rangle (G^{(m)}\otimes 1),\\
          (H^{(m)}\otimes 1)\langle \widetilde{\mathcal{S}}_{1},\ldots,\widetilde{\mathcal{S}}_{m} \rangle&=
          \langle \widetilde{\mathcal{S}}_m,\ldots,\widetilde{\mathcal{S}}_1 \rangle(H^{(m)}\otimes 1).
    \end{align*}
\end{lemma}

\begin{proof}
    The Yang-Baxter equation \eqref{YB-1} implies for $i<j<k$,
    \begin{gather}\label{TB-2}
        R_{ij}R_{ik}R_{jk}=R_{jk}R_{ik}R_{ij}.
    \end{gather}
    Applying the partial transposition $\iota_i,\iota_j$ to both sides of the relation \eqref{TB-2} and replacing $(u_i,u_j)$ with $(-u_i,-u_j)$ to get
    \begin{gather}\label{TB-3}
        R_{ij}R_{ik}^{\iota}R_{jk}^{\iota}=R_{jk}^{\iota}R_{ik}^{\iota}R_{ij},
    \end{gather}
    since $(R_{ij}^{\iota})^{\iota}=R_{ij}$ and $R(u)R(-u)=1-u^{-2}$. Furthermore, by applying $\iota_i$ to both sides of the relation \eqref{TB-2} and replacing $u_i$ with $-u_i$, we also have
    \begin{gather}\label{TB-4}
        R_{jk}R_{ij}^{\iota_i}R_{ik}^{\iota_i}=R_{ik}^{\iota_i}R_{ij}^{\iota_i}R_{jk}.
    \end{gather}
    Since $P_{ij}R_{jk}P_{ij}=R_{ik}(u_j-u_k)$ and
          $P_{ij}R_{ik}^{\iota_i}P_{ij}=R_{jk}^{\iota_j}(-u_i-u_k)$,
          one gets from \eqref{TB-4},
    \begin{gather}\label{TB-5}
        R_{ik}R_{ij}^{\iota_j}R_{jk}^{\iota_j}=R_{jk}^{\iota_j}R_{ij}^{\iota_j}R_{ik}.
    \end{gather}
    Then by \eqref{TB-3}, \eqref{TB-5} and
    \begin{gather*}
        R_{ij}\mathcal{S}_iR_{ij}^{\iota}\mathcal{S}_j=\mathcal{S}_jR_{ij}^{\iota}\mathcal{S}_iR_{ij},
    \end{gather*}
    we obtain
    \begin{align*}
        R_{im}\cdots R_{i,i+1}\langle \mathcal{S}_{i},\ldots,\mathcal{S}_{m} \rangle 
        =\langle \mathcal{S}_{m},\ldots,\mathcal{S}_{i} \rangle R_{im}\cdots R_{i,i+1}.
    \end{align*}
    It proves the first relation of this lemma. The second one is similar.

\end{proof}

From here to the end of this article, we always set $\overline{i}=i+M$ for $i\in \{1,\ldots,N\}$. Denote $H^{(i,p)}$ as the image of the symmetrizer corresponding to the subset of indices $\left\{\overline{i},\ldots,\overline{p}\right\}$. We identify $H^{(i,p)}$ with
$$1^{\otimes (M+i-1)}\otimes H^{(i,p)}\otimes 1^{\otimes (N-p+1)}$$
as an element of $\left(\operatorname{End}V\right)^{\otimes (M+N)}\otimes \mathcal{Y}^{tw}\left(\mathfrak{osp}_{M|N}\right)$.
Observe that $H^{(1,N)}=1^{\otimes M}\otimes H^{(N)}\otimes 1$.

\begin{lemma}\label{product:H}
    It holds for $1\leqslant i<j<q<p\leqslant N$,
    \begin{gather*}
        H^{(i,p)}\Im =\frac{1}{(q-j+1)!}H^{(i,p)}H^{(j,q)}\Im .
    \end{gather*}
\end{lemma}

Lemma \ref{GH-RS} allows us to define
\begin{gather*}
\Im A^{(M|N)}\mathfrak{B}^{tw}_{M|N}(u)\Im=\Im A^{(M|N)} \langle \mathcal{S}_{1},\ldots,\mathcal{S}_{M} \rangle \cdot
         \langle \widetilde{\mathcal{S}}_{M+1},\ldots,\widetilde{\mathcal{S}}_{M+N} \rangle \Im
\end{gather*}
for
\begin{equation*}
u_i=\begin{cases}
  w_i, &\text{if}~~i\in I_+, \\
  w_i-\frac{M-N}{2},      &\text{if}~~i\in I_-.
\end{cases}
\end{equation*}

\begin{proposition}The following equality holds in $\mathcal{Y}^{tw}(\mathfrak{osp}_{M|N})\left[\left[ u^{-1}\right]\right]$,
    \begin{gather*}
         \mathfrak{B}^{tw}_{M|N}(u)=\frac{2u-M-N-1}{2u-M-1}\cdot \mathfrak{B}_{M|N}(u)\mathfrak{P}_{M|N}(u).
    \end{gather*}
\end{proposition}

\begin{proof}
Put
\begin{gather*}
    \Im_+=\mathcal{I}_1\mathcal{I}_2\cdots \mathcal{I}_M,\qquad \Im_-=\mathcal{J}_{M+1}\mathcal{J}_{M+2}\cdots \mathcal{J}_{M+N}.
\end{gather*}
Note that $\Im=\Im_+\Im_-$.
    Define the formal power series
    \begin{gather*}
        \mathfrak{B}_+^{tw}(u),\ \mathfrak{B}_-^{tw}(u)\ \in\mathcal{Y}^{tw}\left(\mathfrak{osp}_{M|N}\right)\left[\left[u^{-1}\right]\right]
    \end{gather*}
    determined by
    \begin{align*}
        \Im_+ \left(G^{(M)}\otimes 1^{\otimes (N+1)}\right)\mathfrak{B}_+^{tw}(u)\Im_+
           &=\Im_+ \left(G^{(M)}\otimes 1^{\otimes (N+1)}\right)\langle \mathcal{S}_{1},\ldots,\mathcal{S}_{M} \rangle\Im_+, \\
        \Im_- \left(1^{\otimes M}\otimes H^{(N)}\otimes 1\right)\mathfrak{B}_-^{tw}(u)\Im_-
           &=\Im_-  \left(1^{\otimes M}\otimes H^{(N)}\otimes 1\right)\langle \widetilde{\mathcal{S}}_{\overline{1}},\ldots,\widetilde{\mathcal{S}}_{\overline{N}} \rangle\Im_- .
    \end{align*}
    Since $\Im_-  \left(1^{\otimes M}\otimes H^{(N)}\otimes 1\right)$ commutes with $\langle S_{1},\ldots,S_{M} \rangle$, we have
    \begin{gather*}
        \mathfrak{B}^{tw}_{M|N}(u)=\mathfrak{B}_+^{tw}(u)\mathfrak{B}_-^{tw}(u).
    \end{gather*}

    We calculate $\mathfrak{B}_-^{tw}(u)$ first. Observe that  $\widetilde{\mathcal{S}}_i(-w_i)=\widetilde{T}_i^{\iota}(w_i)\widetilde{T}_i(-w_i)$.
    Applying $\iota_1$ and $P(\,\cdot\,)P$ to \eqref{Y1} and replacing $v$ with $-v$, we deduce
    \begin{gather*}
        T_2^{\iota}(u)R_{12}^{\iota}(u+v)T_1(-v)=T_1(-v)R_{12}^{\iota}(u+v)T_2^{\iota}(u).
    \end{gather*}
    By $R^{\iota}(u)R^{\iota}(-u+M-N)=1$, one has
    \begin{gather*}
        \widetilde{T}_1(-v)R_{12}^{\iota}(-u-v+M-N)\widetilde{T}_2^{\iota}(u)=\widetilde{T}_2^{\iota}(u) R_{12}^{\iota}(-u-v+M-N)\widetilde{T}_1(-v).
    \end{gather*}
    For the identity matrix $E$, set
    $$\langle E_{\overline{1}},\ldots,E_{\overline{N}} \rangle=(R_{\overline{1}\,\overline{2}}^{\iota}\cdots R_{\overline{1}\,\overline{N}}^{\iota})(R_{\overline{2}\,\overline{3}}^{\iota}\cdots R_{\overline{2}\,\overline{N}}^{\iota})\cdots R_{\overline{N-1}\,\overline{N}}^{\iota}.$$
    Then
    \begin{align*}
        H^{(1,N)}\langle \widetilde{\mathcal{S}}_{\overline{1}},\ldots,\widetilde{\mathcal{S}}_{\overline{N}} \rangle
           &= H^{(1,N)}\widetilde{\mathcal{S}}_{\overline{1}}(R_{\overline{1}\,\overline{2}}^{\iota}\cdots R_{\overline{1}\,\overline{N}}^{\iota})\widetilde{\mathcal{S}}_{\overline{2}}(R_{\overline{2}\,\overline{3}}^{\iota}\cdots R_{\overline{2}\,\overline{N}}^{\iota})\cdots \widetilde{\mathcal{S}}_{\overline{N}}, \\
           &=H^{(1,N)}\widetilde{T}_{\overline{1}}^{\iota}\widetilde{T}_{\overline{1}}(R_{\overline{1}\,\overline{2}}^{\iota}\cdots R_{\overline{1}\,\overline{N}}^{\iota})\widetilde{T}_{\overline{2}}^{\iota}\widetilde{T}_{\overline{2}}(R_{\overline{2}\,\overline{3}}^{\iota}\cdots R_{\overline{2}\,\overline{N}}^{\iota})\cdots \widetilde{T}_{\overline{N}}^{\iota}\widetilde{T}_{\overline{N}} \\
           &=H^{(1,N)}\widetilde{T}_{\overline{1}}^{\iota}\cdots \widetilde{T}_{\overline{N}}^{\iota}\times
               \langle E_{\overline{1}},\ldots,E_{\overline{N}} \rangle
              \times \widetilde{T}_{\overline{1}} \cdots \widetilde{T}_{\overline{N}} \\
            &=\widetilde{T}_{\overline{N}}^{\iota}\cdots \widetilde{T}_{\overline{1}}^{\iota}\times H^{(1,N)}
               \langle E_{\overline{1}},\ldots,E_{\overline{N}} \rangle
              \times \widetilde{T}_{\overline{1}} \cdots \widetilde{T}_{\overline{N}},
    \end{align*}
    for $R_{ij}^{\iota}=R_{i,j}^{\iota}(-w_i-w_j+M-N)$.

    Since
    $$H^{(1,N)}\langle E_{\overline{1}},\ldots,E_{\overline{N}} \rangle=\langle E_{\overline{N}},\ldots,E_{\overline{1}} \rangle H^{(1,N)},$$
    and its action on $\left(\mathbb{C}^{0|N}\right)^{\otimes N}$ is one-dimensional.
    Our next step is to determine the coefficient of
    $$\Im H^{(1,N)}\langle E_{\overline{1}},\ldots,E_{\overline{N}} \rangle v,$$
    for $$v=e_1\otimes\cdots\otimes e_M\otimes e_{\overline{1}}\otimes \cdots\otimes e_{\overline{N}}.$$
    Here, for brevity, we set $\overline{i}'=(\overline{i})'$.
    Denote
    \begin{gather*}
        v_i=e_1\otimes \cdots \otimes e_M \otimes e_{\overline{1}}\otimes \cdots e_{\overline{i}}\otimes e_{\overline{i}'}\otimes \cdots \otimes e_{\overline{i+1}}\otimes \cdots \otimes e_{\overline{N}},
    \end{gather*}
    which is obtained from $v$ by exchanging $e_{\overline{i+1}}$ and $e_{\overline{i}'}$.
    Observe that
    \begin{gather*}
        H^{(i+1,N)}v=-H^{(i+1,N)}v_i,
    \end{gather*}
    and
    \begin{align*}
        \Im H^{(i,N)}R_{\overline{i}\,\overline{i+1}}^{\iota}(-u)\cdots R_{\overline{i}\,\overline{N}}^{\iota}(-u)v_i
         =\Im H^{(i,N)}R_{\overline{i}\,\overline{i}'}^{\iota}(-u)v_i=-\Im \frac{u-2}{u}H^{(i,N)}v.
    \end{align*}
    Then by lemma \ref{product:H}, we have
    \begin{align*}
        &\Im H^{(1,N)}\langle E_{\overline{1}},\ldots,E_{\overline{N}} \rangle v\\
        =&\Im H^{(1,N)}(R_{\overline{1}\,\overline{2}}^{\iota}\cdots R_{\overline{1}\,\overline{N}}^{\iota})(R_{\overline{2}\,\overline{3}}^{\iota}\cdots R_{\overline{2}\,\overline{N}}^{\iota})\cdots (R_{\overline{n}\,\overline{n+1}}^{\iota}\cdots R_{\overline{n}\,\overline{N}}^{\iota}) v \\
        =&\frac{1}{(N-1)!} \cdots \frac{1}{(n+1)!} \Im H^{(1,N)}(R_{\overline{1}\,\overline{2}}^{\iota}\cdots R_{\overline{1}\,\overline{N}}^{\iota})H^{(2,N)}(R_{\overline{2}\,\overline{3}}^{\iota}\cdots R_{\overline{2}\,\overline{N}}^{\iota})\cdots H^{(n,N)}(R_{\overline{n}\,\overline{n+1}}^{\iota}\cdots R_{\overline{n}\,\overline{N}}^{\iota})v \\
    =&\frac{w_{\overline{1}}+w_{\overline{2}}-M+N-2}{w_{\overline{1}}+w_{\overline{2}}-M+N}\times \cdots \times \frac{w_{\overline{n}}+w_{\overline{n+1}}-M+N-2}{w_{\overline{n}}+w_{\overline{n+1}}-M+N} \Im H^{(1,N)}v \\
    =&\frac{2u-M-N-1}{2u-M-1} \Im H^{(1,N)}v.
    \end{align*}
    Here $w_{\overline{i}}+w_{\overline{i+1}}=2u-2N-1+2i$ for each $i\in\{1,\ldots,n\}$.

    In addition, since
    \begin{gather*}
        (1-P)R^{\iota}(u)=(1-P)(1+P^{\iota}u^{-1})=(1-P),
    \end{gather*}
    it follows that
    \begin{gather*}
        G^{(1,M)}(R_{12}^{\iota}\cdots R_{1M}^{\iota})(R_{23}^{\iota}\cdots R_{2M}^{\iota})\cdots R_{M-1,M}^{\iota}=G^{(1,M)}.
    \end{gather*}
    This implies
    \begin{gather*}
        G^{(1,M)}\langle \mathcal{S}_{1},\ldots,\mathcal{S}_{M} \rangle=G^{(1,M)}T_1\cdots T_M T_{1}^{\iota}\cdots T_{M}^{\iota},
    \end{gather*}
    where
    $$G^{(1,M)}=G^{(M)}\otimes 1^{\otimes (N+1)},\ \mathcal{S}_i=\mathcal{S}_i(w_i),\
        T_i=T_i(w_i),\ T_i^{\iota}=T_i^{\iota}(-w_i), R_{ij}^{\iota}=R_{ij}^{\iota}(-w_i-w_j).$$
        Hence, we complete the proof.

\end{proof}
We call $\mathfrak{B}_{M|N}^{tw}(u)$ the \textit{quantum Berezinian} of $\mathcal{Y}^{tw}\left(\mathfrak{osp}_{M|N}\right)$. The next proposition is the \textit{quantum Liouville formula} for $\mathcal{Y}^{tw}\left(\mathfrak{osp}_{M|N}\right)$, which is a direct conclusion by the definition of  $\mathfrak{B}_{M|N}^{tw}(u)$.

\begin{proposition}
    The following equation holds in $\mathcal{Y}^{tw}(\mathfrak{osp}_{M|N})\left[\left[ u^{-1}\right]\right]$,
    \begin{gather}\label{Q:Liou:tw}
       \mathfrak{z}(u)=\frac{(2u-M-N-1)(2u-M+1)}{(2u-M-N+1)(2u-M-1)}\cdot \frac{\mathfrak{B}^{tw}(u+1)}{\mathfrak{B}^{tw}(u)}.
    \end{gather}
\end{proposition}

\begin{remark}
    When $N$ goes to 0, the scalar in \eqref{Q:Liou:tw} is equal to 1, which provides a quantum Liouville formula for the twisted Yangian $\mathrm{Y}^{tw}(\mathfrak{o}_M)$, cf. \cite[Section 2.11]{Mo07}.
\end{remark}

\begin{corollary}
    The coefficients of the series $\mathfrak{B}_{M|N}^{tw}(u)$ are central elements of $\mathcal{Y}^{tw}\left(\mathfrak{osp}_{M|N}\right)$. Moreover, they generate the center of $\mathcal{Y}^{tw}\left(\mathfrak{osp}_{M|N}\right)$.
\end{corollary}

\subsection{The Explicit formula for $\mathfrak{B}_{M|N}^{tw}(u)$}\label{se:quantumbereianian:2}

In this section, we present an explicit formula for the quantum Berezinian $\mathfrak{B}_{M|N}^{tw}(u)$ of the twisted super Yangian $\mathcal{Y}^{tw}\left(\mathfrak{osp}_{M|N}\right)$. Prior to this, we need to introduce a series of projections of permutation groups:
\begin{gather*}
    \Omega_p:\ \mathfrak{S}_p\rightarrow \mathfrak{S}_p\quad \sigma\mapsto \sigma',
\end{gather*}
which is defined inductively as follows.

Consider an index set $\{k_1,\ldots,k_p\}$ for $p\geqslant2$ with the lexicographical order. The action of $\Omega_p$ on an ordered pair $(k_a,k_b)$ is determined by the following rule:
\begin{equation*}
    \begin{aligned}
    (k_a,k_b)\quad&\mapsto\quad (k_b,k_a), \\
    (k_a,k_p)\quad&\mapsto\quad (k_{p-1},k_a),  \\
    (k_p,k_b)\quad&\mapsto\quad (k_b,k_{p-1}),  \\
    (k_{p-1},k_p)\quad&\mapsto\quad (k_{p-1},k_{p-2}), \\
    (k_p,k_{p-1})\quad&\mapsto\quad (k_{p-1},k_{p-2}).
    \end{aligned}\quad
    \begin{aligned}
        &\text{if }~~a,b<p, \\
        &\text{if }~~a<p-1, \\
        &\text{if }~~b<p-1, \\
        & \\
        &
    \end{aligned}
\end{equation*}
For the special case $p=2$, we define the action of $\Omega_2$ such that it maps both $(k_1,k_2)$ and $(k_2,k_1)$ onto the element $k_1$. These projections are clearly well-defined.

Let $\sigma=(k_{a_1},\ldots,k_{a_p})$ be a permutation of the indices $\{k_1,\ldots,k_p\}$. Its image
under the mapping $\Omega_p$ is the permutation $\sigma'=\{k_{b_1},\ldots,k_{b_{p-1}},k_p\}$, where
the pair $(k_{b_1},k_{b_{p-1}})$ is the image of the ordered pair $(k_{a_1},k_{a_p})$. Subsequently, the pair $(k_{b_2},k_{b_{p-2}})$ is determined as the image of $(k_{a_2},k_{a_{p-1}})$, which is defined on the set of ordered pairs of elements obtained from $\{k_1,\ldots,k_p\}$ by deleting $k_{a_1}$ and $k_{a_p}$. The procedure is completed in a similar fashion by consecutively identifying the pairs $(k_{b_i},k_{b_{p-i}})$.

Define the following formal series in $\mathcal{Y}^{tw}\left(\mathfrak{osp}_{M|N}\right)\left[\left[u^{-1}\right]\right]$:
\begin{align*}
    \mathcal{S}_{\ell_1,\ldots,\ell_p}^{k_1,\ldots,k_p}(u),\quad
           \check{\mathcal{S}}_{\ell_1,\ldots,\ell_{p-1},c}^{k_1,\ldots,k_p}(u),
\end{align*}
for $p\leqslant N$ and each $k_i,\ell_i,c\in\{\overline{1},\ldots,\overline{N}\}$ via
\begin{align*}
    &\Im H^{(1,p)}\langle \widetilde{\mathcal{S}}_{\overline{1}},\ldots,\widetilde{\mathcal{S}}_{\overline{p}} \rangle \Im
        \left(e_{\ell_1}\otimes\cdots\otimes  e_{\ell_p}\right)
       =\sum  \mathcal{S}_{\ell_1,\ldots,\ell_p}^{k_1,\ldots,k_p}(u)\left(e_{k_1}\otimes\cdots\otimes e_{k_p}\right), \\
    &\Im H^{(1,p)}\langle \widetilde{\mathcal{S}}_{\overline{1}},\ldots,\widetilde{\mathcal{S}}_{\overline{p}} \rangle \mathcal{S}_{\overline{p}}(-w_{\bar{p}})\Im
       \left( e_{\ell_{1}}\otimes\cdots\otimes  e_{\ell_{p}}\otimes e_c\right) =\sum  \check{\mathcal{S}}_{\ell_1,\ldots,\ell_{p-1},c}^{k_1,\ldots,k_p}(u) \left(e_{k_1}\otimes\cdots\otimes e_{k_p}\right).
\end{align*}
Then we immediately have
\begin{lemma}\label{S:check}
  The following equation holds in $\mathcal{Y}^{tw}(\mathfrak{osp}_{M|N})\left[\left[ u^{-1}\right]\right]$,
  \begin{gather*}
      \mathcal{S}_{\ell_1,\ldots,l_p}^{k_1,\ldots,k_p}(u)
      =\sum_{c}\check{\mathcal{S}}_{\ell_1,\ldots,l_{p-1},c}^{k_1,\ldots,k_p}(u)\tilde{\mathfrak{s}}_{c,\ell_p}(-w_{\bar{p}}).
  \end{gather*}
\end{lemma}

\begin{lemma}
    For any permutation $\sigma\in\mathfrak{S}_{p}$, $\sigma'\in\mathfrak{S}_{p-1}$, the following identities hold in $\mathcal{Y}^{tw}(\mathfrak{osp}_{M|N})\left[\left[ u^{-1}\right]\right]$,
    \begin{align}\label{permu:ST1}
        \mathcal{S}_{\ell_1,\ldots,\ell_{p-1},\ell_p}^{k_{\sigma(1)},\ldots,k_{\sigma(p)}}(u)
            &=\epsilon(\sigma)\mathcal{S}_{\ell_1,\ldots,\ell_{p-1},\ell_p}^{k_1,\ldots,k_p}(u), \\ \label{permu:ST2}
        \mathcal{S}_{\ell_{\sigma(1)},\ldots,\ell_{\sigma(p)}}^{k_1,\ldots,k_p}(u)
            &=\epsilon(\sigma)\mathcal{S}_{\ell_1,\ldots,\ell_{p-1},c}^{k_1,\ldots,k_p}(u), \\ \label{permu:ST3}
        \check{\mathcal{S}}_{\ell_1,\ldots,\ell_{p-1},c}^{k_{\sigma(1)},\ldots,k_{\sigma(p)}}(u)
            &=\epsilon(\sigma)\check{\mathcal{S}}_{\ell_1,\ldots,\ell_{p-1},c}^{k_1,\ldots,k_p}(u), \\ \label{permu:ST4}
        \check{\mathcal{S}}_{\ell_{\sigma'(1)},\ldots,\ell_{\sigma'(p-1)},c}^{k_1,\ldots,k_p}(u)
            &=\epsilon(\sigma')\check{\mathcal{S}}_{\ell_1,\ldots,\ell_{p-1},c}^{k_1,\ldots,k_p}(u).
    \end{align}
\end{lemma}

\begin{proof}
We only need to verify the last equation \eqref{permu:ST4}, as the others \eqref{permu:ST1}, \eqref{permu:ST2}, and \eqref{permu:ST3} are obvious from their definitions. By Lemma \ref{product:H}, we have
\begin{align*}
    \Im H^{(1,p)}\langle \widetilde{\mathcal{S}}_{\overline{1}},\ldots,\widetilde{\mathcal{S}}_{\overline{p}} \rangle  \mathcal{S}_{\overline{p}}(-w_{\bar{p}})\Im
    &=\Im \langle \widetilde{\mathcal{S}}_{\overline{p}},\ldots,\widetilde{\mathcal{S}}_{\overline{1}} \rangle\cdot \frac{1}{(p-1)!} H^{(1,p)}H^{(1,p-1)}\cdot \mathcal{S}_{\overline{p}}(-w_{\bar{p}})\Im \\
    &=\Im \frac{1}{(p-1)!} H^{(1,p)}\langle \widetilde{\mathcal{S}}_{\overline{1}},\ldots,\widetilde{\mathcal{S}}_{\overline{p}} \rangle\cdot \mathcal{S}_{\overline{p}}(-w_{\bar{p}}) H^{(1,p-1)}\Im ,
\end{align*}
which proves \eqref{permu:ST4}.

\end{proof}

\begin{corollary}\label{Formula:S}
    The following identities hold for $a\in\{1,\ldots,p\}$ and $b\in\{1,\ldots,p-1\}$ in $\mathcal{Y}^{tw}(\mathfrak{osp}_{M|N})\left[\left[ u^{-1}\right]\right]$,
    \begin{align*}
        \mathcal{S}_{\ell_1,\ldots,\ell_p}^{k_1,\ldots,k_p}(u)
          &=(-1)^{p-a}\,\mathcal{S}_{\ell_1,\ldots,\hat{\ell}_a,\ldots,\ell_p,\ell_a}^{k_1,\ldots,k_p}(u), \\
        \check{\mathcal{S}}_{\ell_1,\ldots,\ell_{p-1},c}^{k_1,\ldots,k_p}(u)
          &=(-1)^{p-a}\,\check{\mathcal{S}}_{\ell_1,\ldots,\ell_{p-1},c}^{k_1,\ldots,\hat{k}_a,\ldots,k_p,k_a}(u), \\
        \check{\mathcal{S}}_{\ell_1,\ldots,\ell_{p-1},c}^{k_1,\ldots,k_p}(u)
          &=(-1)^{b-1}\,\check{\mathcal{S}}_{\ell_b,\ell_1,\ldots,\hat{\ell}_b,\ldots,\ell_{p-1},c}^{k_1,\ldots,k_p}(u).
    \end{align*}
\end{corollary}

\begin{lemma}
    The following equation holds in $(\operatorname{End}V)^{\otimes (M+N)}$,
    \begin{gather*}
        H^{(2,N)}R_{M+1,M+2}^{\iota}\cdots R_{M+1,M+N}^{\iota}=\left(1+\mathfrak{c}_u(Q_{12}+\cdots +Q_{1N})\right)H^{(2,N)},
    \end{gather*}
    where $\mathfrak{c}_u=\frac{1}{2u-M-N+1}$, $R_{M+1,M+i}^{\iota}=R_{M+1,M+i}^{\iota}(-2u+M+N+1-i)$ and for $i<j$,
    \begin{gather*}
        Q_{ij}=P_{M+i,M+j}^{\iota}.
    \end{gather*}
\end{lemma}

\begin{proof}
    By the action of $\pi_N$, we have for $i_1>i_2>\ldots>i_m>1$ and $\{i_1,\ldots,i_m\}\subset\{2,\ldots,N\}$,
\begin{gather*}
P_{\overline{1}\,\overline{i_1}}P_{\overline{1}\,\overline{i_2}}\cdots P_{\overline{1}\,\overline{i_m}}=\pi_{N}\left((\,\overline{1} \,\overline{i_m}\,\cdots \,\overline{i_1}\,)\right)=P_{\overline{i_2}\,\overline{i_1}}\cdots P_{\overline{i_m}\,\overline{i_{m-1}}}P_{\overline{1}\,\overline{i_1}}.
\end{gather*}
Then
\begin{gather}\label{Q:eq1}
H^{(2,N)}P_{\overline{1}\,\overline{i_1}}P_{\overline{1}\,\overline{i_2}}\cdots P_{\overline{1}\,\overline{i_m}}=H^{(2,N)}P_{\overline{1}\,\overline{i_1}},
\end{gather}
since
$$P_{i,j}\left(1^{\otimes (i-1)}\otimes e_{a_i}\otimes 1^{\otimes (j-i-1)}\otimes e_{a_j}\otimes 1^{\otimes (M+N-j)}\right)=-1^{\otimes (i-1)}\otimes e_{a_j}\otimes 1^{\otimes (j-i-1)}\otimes e_{a_i}\otimes 1^{\otimes (M+N-j)}$$
for $|i|=|j|=1$. Applying the partial transposition $\iota_{\overline{1}}$ to relation \eqref{Q:eq1}, we obtain
\begin{gather*}
H^{(2,N)}Q_{1i_m}\cdots Q_{1i_2}Q_{1i_1}=H^{(2,N)}Q_{1i_1}.
\end{gather*}
It follows that
\begin{gather*}
H^{(2,N)}R_{\overline{1}\,\overline{2}}^{\iota}\cdots R_{\overline{1}\,\overline{N}}^{\iota}=H^{(2,N)}\left(1+\mathfrak{c}_u(Q_{12}+\cdots +Q_{1N})\right),
\end{gather*}
by the identity
$$\frac{1}{u'+r}\prod_{p=1}^{r-1}\left(1+\frac{1}{u'+p}\right)=\frac{1}{u'+1},$$
for $u'=2u-M-N$ and $r\in\{1,\ldots,N-1\}$.

Moreover, applying the partial transposition $\iota_{\bar{1}}$ to
\begin{gather*}
    H^{(2,N)}\left(P_{\overline{1}\,\overline{2}}+\cdots+P_{\overline{1}\,\overline{N}}\right)
      =\left(P_{\overline{1}\,\overline{2}}+\cdots+P_{\overline{1}\,\overline{N}}\right)H^{(2,N)},
\end{gather*}
then we get
\begin{gather*}
    H^{(2,N)}\left(Q_{12}+\cdots+Q_{1N}\right)
      =\left(Q_{12}+\cdots+Q_{1N}\right)H^{(2,N)},
\end{gather*}
which proves the lemma.

\end{proof}

For $i,j\in\left\{\overline{1},\ldots,\overline{N}\right\}$, we define the series in $\mathcal{Y}^{tw}\left(\mathfrak{osp}_{M|N}\right)\left[\left[ u^{-1}\right]\right]$:
\begin{gather*}
    \mathfrak{s}_{ij}^{\natural}(-u)=(1+\mathfrak{c}_{u+N})\tilde{\mathfrak{s}}_{ij}(u)+\mathfrak{c}_{u+N}\theta_i\theta_j\tilde{\mathfrak{s}}_{j'i'}(u).
\end{gather*}
In particular, $\mathfrak{s}_{ii'}^{\natural}(-u)=\tilde{\mathfrak{s}}_{ii'}(u)$.
\begin{proposition}
    Suppose that $k_1<\cdots<k_{p}$, $\ell_2<\cdots<\ell_{p-1}$, $\ell_1\in\{k_1,\ldots,k_m\}$, $c\notin\{\ell_2,\ldots,\ell_{p-1}\}$ with parities $|k_i|=|\ell_i|=|c|=1$. Then
    \begin{align*}
        &\check{\mathcal{S}}_{\ell_1,\ldots,l_{p-1},c}^{k_1,\ldots,k_p}(u)=0,\quad \text{if }c\notin\{k_1,\ldots,k_p\}, \\
        &\check{\mathcal{S}}_{\ell_1,\ldots,l_{p-1},c}^{k_1,\ldots,k_p}(u)=
           \sum_{r=1}^{p-1} (-1)^{r-1}\mathfrak{s}_{k_r\ell_1}^{\natural}(u-N) \mathcal{S}_{\ell_1,\ldots,l_p}^{k_1,\ldots,\hat{k}_r,\ldots,k_p}(u+1),\quad \text{if}~~c=k_p,
    \end{align*}
    where $\hat{k}$ means the corresponding index $k$ omitted.
\end{proposition}

\begin{proof}
 By definition,  we get the following relation summed over $\{k_1,\ldots,k_p\}\subset\{\bar{1},\ldots,\bar{N}\}$,
 \begin{align*}
     \Im H^{(1,p)}\widetilde{\mathcal{S}}_{\overline{1}}(R_{\overline{1}\,\overline{2}}^{\iota}\cdots R_{\overline{1}\,\overline{p}}^{\iota})\widetilde{\mathcal{S}}_{\overline{2}}(R_{\overline{2}\,\overline{3}}^{\iota}\cdots R_{\overline{2}\,\overline{p}}^{\iota})\cdots \widetilde{\mathcal{S}}_{\overline{p-1}}\cdots R_{\overline{p-1}\,\overline{p}}^{\iota}\left(e_{\ell_{1}}\otimes \cdots\otimes e_{\ell_{p-1}}\otimes e_c\right).
 \end{align*}
 Since $H^{(2,p)}$ commutes with $\widetilde{\mathcal{S}}_{\overline{1}}$ and $Q_{12}+\cdots+Q_{1p}$, it is equal to
 \begin{equation}\label{Q:eq2}
 \begin{split}
     &\Im e_+ +\frac{1}{(p-2)!}\Im H^{(1,p)}\widetilde{\mathcal{S}}_{\overline{1}}\left(1+\mathfrak{c}_u(Q_{12}+\cdots+Q_{1p})\right) \\
     &\qquad\qquad\qquad\qquad\qquad \times \sum_{h_2,\ldots,h_{p-1}}\mathcal{S}_{\ell_2,\ldots,\ell_{p-1}}^{h_2,\ldots,h_{p-1}}(u+1) \left(e_{\ell_1}\otimes e_{h_2}\otimes \cdots\otimes e_{h_{p-1}}\otimes e_{c}\right),
 \end{split}
 \end{equation}
 for $e_+\in V^{\otimes p}\setminus (\mathbb{C}^{0|N})^{\otimes p}$. Note that $\Im e_+=0$. We are primarily interested in the coefficient of $e_{k_1}\otimes \cdots\otimes e_{k_{p-1}}\otimes e_{k_p}$ in the following expression
 \begin{gather}\label{Q:eq3}
     \Im H^{(1,p)}\widetilde{\mathcal{S}}_{\bar{1}}\left(1+\mathfrak{c}_u(Q_{1,2}+\cdots+Q_{1,p})\right)
        \sum_{h_2,\ldots,h_{p-1}}\left(e_c\otimes e_{h_{p-1}}\otimes \cdots\otimes e_{h_2}\otimes e_{\ell_1}\right).
 \end{gather}

 If $c\notin\{k_1,\ldots,k_p\}$, then we also have $c\neq \ell_1'$, which forces
 \begin{gather*}
   Q_{1p}\left(e_{\ell_1}\otimes e_{h_2}\otimes \cdots\otimes e_{h_{p-1}}\otimes e_c\right)=0.
 \end{gather*}
 Observe that the segment following $\Im H^{(1,p)}$ in \eqref{Q:eq2} must converse $e_c$, which implies that the coefficient associated to $\{k_1,\ldots,k_p\}$ is zero.

 Next, suppose $c=k_p=\ell_1'$. The summation can only be taken over $h_j\neq \ell_1'$ for $j=2,\ldots,p-1$, otherwise $H^{(2,p)}$ annihilates the vector $e_{h_2}\otimes\cdots\otimes e_{h_{p-1}}\otimes e_c$. So
 \begin{gather*}
    Q_{1,j}\left(e_{\ell_1}\otimes e_{h_2}\otimes\cdots\otimes e_{h_{p-1}}\otimes e_c \right)=0,\quad \text{for }j=2,\ldots,p-1.
 \end{gather*}
 Hence \eqref{Q:eq3} is equal to
 \begin{gather*}
     \Im H^{(1,p)}\widetilde{\mathcal{S}}_{\bar{1}}\left(1+\mathfrak{c}_uQ_{1p}\right)
        \sum_{h_2,\ldots,h_{p-1}}\left(e_{\ell_1}\otimes e_{h_2}\otimes\cdots\otimes e_{h_{p-1}}\otimes e_c \right).
 \end{gather*}
 Expand this equation, we get the form
 \begin{equation}\label{Q:eq4}
 \begin{split}
     &\Im H^{(1,p)}\bigg(\sum_{h_1,\ldots,h_{p-1}}\tilde{\mathfrak{s}}_{h_1k_p'}(-u+N)\left(e_{h_1}\otimes \cdots\otimes e_{h_{p-1}}\otimes e_{k_p}\right) \\
     &\qquad\qquad\qquad +\mathfrak{c}_u\cdot \sum_{h_1,\ldots,h_{p}}\theta_{h_p}\theta_{k_p}\tilde{\mathfrak{s}}_{h_1h_p'}(-u+N)\left(e_{h_1}\otimes \cdots\otimes e_{h_{p-1}}\otimes e_{h_p}\right) \bigg).
 \end{split}
 \end{equation}
 Then the coefficient of $e_{k_1}\otimes\cdots\otimes  e_{k_{p-1}}\otimes e_{k_p}$ in \eqref{Q:eq3} has the form
 \begin{gather*}
 (p-2)!\sum_{r=1}^{p-1}\left((1+\mathfrak{c}_u)\tilde{\mathfrak{s}}_{k_r\ell_1}(-u+N)+\mathfrak{c}_u\theta_{k_r}\theta_{\ell_1}\tilde{\mathfrak{s}}_{\ell_1'k_r'}(-u+N)\right)=(p-2)!\sum_{r=1}^{p-1}\mathfrak{s}_{k_r\ell_1}^{\natural}(u-N).
 \end{gather*}

 Now let $c=k_p$ but $c\neq \ell_1'$. Relation \eqref{Q:eq3} is equal to
 \begin{gather*}
     \Im H^{(1,p)}\widetilde{\mathcal{S}}_{\bar{1}}\left(1+\mathfrak{c}_u(Q_{12}+\cdots+Q_{1p})\right)
        \sum_{h_2,\ldots,h_{p-1}}\left(e_{\ell_1}\otimes e_{h_2}\otimes\cdots\otimes e_{h_{p-1}}\otimes e_c \right).
 \end{gather*}
 Then the coefficient of $e_{k_1}\otimes\cdots\otimes  e_{k_{p-1}}\otimes e_{k_p}$ in \eqref{Q:eq2} has the form
 \begin{align*}
     &\frac{1}{(p-2)!}\sum_{\sigma\in\mathfrak{S}_{p-1}}\epsilon(\sigma)\tilde{\mathfrak{s}}_{k_{\sigma(1)},\ell_1}(-u+N) \mathcal{S}_{\ell_2,\ldots,\ell_{p-1}}^{k_{\sigma(2)},\ldots,k_{\sigma(p-1)}}(u+1) \\
        &-\frac{1}{(p-2)!}\mathfrak{c}_u\sum_{r=2}^{p-1}\sum_{\sigma\in\mathfrak{S}_{m-1}}\epsilon(\sigma)\theta_{\ell_1}\theta_{k_{\sigma(r)}}\tilde{\mathfrak{s}}_{k_{\sigma(1)},k_{\sigma(r)}'}(-u+N)\mathcal{S}_{\ell_2,\ldots,\ell_{p-1}}^{k_{\sigma(2)},\ldots,\hat{k}_{\sigma(r)},\ldots,k_{\sigma(p)}}(u+1) \\
    =&\sum_{r=1}^{p-1}(-1)^{r-1}\tilde{\mathfrak{s}}_{k_{r},\ell_1}(-u+N) \mathcal{S}_{\ell_2,\ldots,\ell_{p-1}}^{k_1,\ldots,\hat{k}_r,\ldots,k_{p-1}}(u+1)-\mathfrak{c}_u\sum_{r\in\{1,\ldots,p-1\},k_{r}\neq \ell_1'}(-1)^r \\
        &\times\left(\theta_{\ell_1}\theta_{\ell_1'}\tilde{\mathfrak{s}}_{k_r,\ell_1}(-u+N)-\theta_{\ell_1}\theta_{k_r}\tilde{\mathfrak{s}}_{\ell_1',k_r'}(-u+N)\right) \mathcal{S}_{\ell_2,\ldots,\ell_{p-1}}^{k_1,\ldots,\hat{k}_{r},\ldots,k_{p-1}}(u+1) \\
    =&(-1)^{r-1}\delta_{k_r\ell_1'}\tilde{\mathfrak{s}}_{\ell_1',\ell_1}(-u+N) \mathcal{S}_{\ell_2,\ldots,\ell_{p-1}}^{k_2,\ldots,\ldots,k_{p-1}}(u+1)+\sum_{r\in\{1,\ldots,p-1\},k_r\neq \ell_1'}(-1)^{r-1} \\
        &\times\left((1+\mathfrak{c}_u)\tilde{\mathfrak{s}}_{k_r,\ell_1}(-u+N) +\mathfrak{c}_u\theta_{\ell_1}\theta_{k_r}\tilde{\mathfrak{s}}_{\ell_1',k_r'}(-u+N)\right)  \mathcal{S}_{\ell_2,\ldots,\ell_{p-1}}^{k_1,\ldots,\hat{k}_{r},\ldots,k_{p-1}}(u+1) \\
    =&\sum_{r=1}^{p-1}\tilde{\mathfrak{s}}_{k_r,\ell_1}^{\natural}(u-N)\mathcal{S}_{\ell_2,\ldots,\ell_{p-1}}^{k_1,\ldots,\hat{k}_{r},\ldots,k_{p-1}}(u+1).
 \end{align*}
\end{proof}

\begin{theorem}\label{main}
     We have the explicit form of $\mathfrak{B}_{M|N}^{tw}(u)$\,{\rm :}
\begin{align*}
\mathfrak{B}_{M|N}^{tw}(u)=
       &\sum_{\sigma\in\mathfrak{S}_M}\operatorname{sgn}(\sigma\sigma')\Big\{ \mathfrak{s}_{\sigma(1),\sigma'(1)}^{\iota}(-u-M+N+1)\cdots \mathfrak{s}_{\sigma(m),\sigma'(m)}^{\iota}(-u-M+m+N) \\
            &\qquad\cdot  \mathfrak{s}_{\sigma(m+1),\sigma'(m+1)}(u+M-m-N-1)\cdots \mathfrak{s}_{\sigma(M),\sigma'(M)}(u-N)\Big\} \\
       &\times\sum_{\sigma\in\mathfrak{S}_N}\operatorname{sgn}(\sigma\sigma') \Big\{\mathfrak{s}_{M+\sigma(1),M+\sigma'(1)}^{\natural}(u-N)\cdots \mathfrak{s}_{M+\sigma(n),M+\sigma'(n)}^{\natural}(u-n-1) \\
       &\qquad \cdot \tilde{\mathfrak{s}}_{M+\sigma(n+1),M+\sigma'(n+1)}(-u+n)\cdots \tilde{\mathfrak{s}}_{M+\sigma(N),M+\sigma'(N)}(-u+1)\Big\},
\end{align*}
for $M=2m$ or $M=2m+1$. Here $\sigma'$ is obtained from $\sigma$ under the map $\Omega_M$ $(\text{resp. }\Omega_N)$ for $\sigma\in \mathfrak{S}_M$ $(\text{resp. }\sigma\in \mathfrak{S}_N)$.
\end{theorem}

\begin{proof}
  We first calculate the explicit form of $\mathcal{S}_{k_1,\ldots,k_{p-1},c}^{k_{1},\ldots,k_{p}}(u)$ for $k_i,c\in\{\bar{1},\ldots,\bar{N}\}$ and $k_1<k_2<\cdots<k_{p-1}<k_p$. Indeed, by Lemma \ref{S:check} and Corollary \ref{Formula:S}, we have
  \begin{equation*}
      \begin{split}
          &\mathcal{S}_{k_1,\ldots,k_{p-1},c}^{k_{1},\ldots,k_{p}}(u)=\sum_{a=1}^p \check{\mathcal{S}}_{k_1,\ldots,k_{p-1},k_a}^{k_{1},\ldots,k_{p}}(u)\mathfrak{s}_{k_a c}(-w_{\bar{p}}) \\
             =&(-1)^{p-2}\check{\mathcal{S}}_{k_{p-1},k_1,\ldots,k_{p-2},k_p}^{k_{1},\ldots,k_{p-1},k_{p}}(u) \mathfrak{s}_{k_p c}(-w_{\bar{p}})+\sum_{a=1}^{p-1}(-1)^{p-1} \check{\mathcal{S}}_{k_a,k_1,\ldots,\hat{k}_a,\ldots,k_{p-1},k_a}^{k_1,\ldots,\hat{k}_a,\ldots,k_p,k_a}(u) \mathfrak{s}_{k_a c}(-w_{\bar{p}}) \\
             =&\mathfrak{s}_{k_{p-1},k_{p-1}}^{\natural}(w_{\bar{1}})\mathcal{S}_{k_1,\ldots,k_{p-2}}^{k_{1},\ldots,k_{p-2}}(u+1) \mathfrak{s}_{k_p c}(-w_{\bar{p}})-\sum_{b=1}^{p-2} \mathfrak{s}_{k_b,k_{p-1}}^{\natural}(w_{\bar{1}}) \mathcal{S}_{k_1,\ldots,\hat{k}_b,\ldots,k_{p-2},k_b}^{k_{1},\ldots,\hat{k}_b,\ldots,k_{p-1}}(u+1) \mathfrak{s}_{k_p c}(-w_{\bar{p}}) \\
             &-\sum_{a=1}^{p-1}\bigg( \mathfrak{s}_{k_pk_a}^{\natural}(w_{\bar{1}}) \mathcal{S}_{k_1,\ldots,\hat{k}_a,\ldots,k_{p-1}}^{k_1,\ldots,\hat{k}_a,\ldots,k_{p-1}}(u+1) \mathfrak{s}_{k_a c}(w_{\bar{p}})\\
             &\qquad\qquad-\sum_{b=1}^{a-1} s_{k_bk_a}^{\natural}(w_{\bar{1}}) \mathcal{S}_{k_1,\ldots,\hat{k}_b,\ldots,\hat{k}_a,\ldots,k_{p-1},k_b}^{k_1,\ldots,\hat{k}_b\ldots,\hat{k}_a,\ldots,k_p}(u+1) \mathfrak{s}_{k_a c}(-w_{\bar{p}}) \\
             &\qquad\qquad\qquad-\sum_{b=a+1}^{p-1} \mathfrak{s}_{k_bk_a}^{\natural}(w_{\bar{1}}) \mathcal{S}_{k_1,\ldots,\hat{k}_a,\ldots,\hat{k}_b,\ldots,k_{p-1},k_b}^{k_1,\ldots,\hat{k}_a\ldots,\hat{k}_b,\ldots,k_p}(u+1) \mathfrak{s}_{k_a c}(-w_{\bar{p}})\bigg).
      \end{split}
  \end{equation*}
  Observe that the number of the terms in the right side of this equation is $p!$. As an example,
  \begin{align*}
      \mathcal{S}_{k_1c}^{k_1k_2}(u)=&\mathfrak{s}_{k_1k_1}^{\natural}(w_{\bar{1}})\mathfrak{s}_{k_2c}(-w_{\bar{2}})-\mathfrak{s}_{k_2k_1}^{\natural}(w_{\bar{1}})\mathfrak{s}_{k_1c}(-w_{\bar{2}}) \\
      =&\mathfrak{s}_{\sigma_1(k_1)k_1}^{\natural}(w_{\bar{1}})\mathfrak{s}_{\sigma_1(k_2)c}(-w_{\bar{2}})-\mathfrak{s}_{\sigma_2(k_1)k_1}^{\natural}(w_{\bar{1}})\mathfrak{s}_{\sigma_2(k_2)c}(-w_{\bar{2}}),
  \end{align*}
  for $\sigma_1=(1)$ and $\sigma_2=(k_1k_2)$.
  Since $$\mathfrak{B}_-^{tw}(u)=\mathcal{S}_{\overline{1},\ldots,\overline{N}}^{\overline{1},\ldots,\overline{N}}(u)=\mathcal{S}_{\overline{2},\ldots,\overline{N-1},\overline{1},\overline{N}}^{\overline{2},\ldots,\overline{N-1},\overline{1},\overline{N}}(u),$$
which is equal to
  \begin{gather*}
      \sum_{\sigma\in \mathfrak{S}_N} \epsilon(\sigma\sigma')\mathfrak{s}_{\overline{\sigma(1)},\overline{\sigma'(1)}}^{\natural}(w_{\overline{1}})\cdots \mathfrak{s}_{\overline{\sigma(n)},\overline{\sigma'(n)}}^{\natural}(w_{\overline{n}})\tilde{\mathfrak{s}}_{\overline{\sigma(n+1)},\overline{\sigma'(n+1)}}(-w_{\overline{n+1}})\cdots \tilde{\mathfrak{s}}_{\overline{\sigma(N)},\overline{\sigma'(N)}}(-w_{\overline{N}})
  \end{gather*}
  by using the recurrence relation above. Here, $\sigma'$ is obtained from $\sigma$ under the map $\Omega_N$.
  The form of $\mathfrak{B}_+^{tw}(u)$ in terms of generators $\mathfrak{s}_{ij}$ can be checked similarly owing to
   $$(1+\frac{1}{2u-1})\mathfrak{s}_{ij}(u)-\frac{1}{2u-1}\theta_i\theta_j\mathfrak{s}_{j'i'}(u)=\theta_i\theta_j\mathfrak{s}_{j'i'}(-u)=\mathfrak{s}_{ij}^{\iota}(-u),\quad i,j\in I_+.$$
\end{proof}

\begin{remark}
\begin{itemize}
    \item[{\rm (1)}] The formula of $\mathfrak{B}^{tw}_{M|N}(u)$ presented in Theorem \ref{main} is a super-version of the Sklyanin determinant described in \cite[Theorem 4.5]{Mo95}. When $N$ is equal to 0, it specializes to the Sklyanin determinant of the twisted Yangian $\mathrm{Y}^{tw}\left(\mathfrak{o}_M\right)$. However, it does not seem to be directly applicable to the case $M=0$, due to the fact that the mapping
    \begin{gather*}
        \mathcal{S}(u)\mapsto \widetilde{\mathcal{S}}(-u-\frac{N}{2})
    \end{gather*}
    does not satisfy the symmetry relation \eqref{Yeqv2}. Thus it is not an automorphism of the twisted Yangian $\mathrm{Y}^{tw}\left(\mathfrak{sp}_N\right)$.
    \item[{\rm (2)}] Briot-Ragoucy in their paper \cite{BR03} defined a ``$-$" version of twisted super Yangian $\mathcal{Y}^{tw}\left(\mathfrak{spo}_{N|M}\right)$ as a coideal of super Yangian $\mathrm{Y}\left(\mathfrak{gl}_{N|M}\right)$. We can also introduce the quantum Berezinian for the superalgebra $\mathcal{Y}^{tw}\left(\mathfrak{spo}_{N|M}\right)$ and obtain its explicit formula using the same method as in this section. When $M$ is equal to 0, the quantum Berezinian for $\mathcal{Y}^{tw}\left(\mathfrak{spo}_{N|M}\right)$ coincides with the Sklyanin determinant of the twisted Yangian $\mathrm{Y}^{tw}\left(\mathfrak{sp}_{N}\right)$ investigated in \cite[Theorem 4.5]{Mo07}.
    \item[{\rm (3)}] Let $\mathcal{G}$ be a supersymmetric block matrix such that $\mathcal{G}_0=B\mathcal{G}B^t$ for an invertible matrix $B$ as mentioned in Proposition \ref{Ytw:eqv}. The explicit form of the quantum Berezinian for twisted super Yangian $\mathrm{Y}^{tw}\big(\mathfrak{osp}_{M|N}\big)$ associated to $\mathcal{G}$ is obtained from $\mathfrak{B}^{tw}_{M|N}(u)$ by replacing $\mathcal{S}(u)$ with $BS(u)B^t\mathcal{G}_0^{-1}$. Here $S(u)$ is the generator matrix of $\mathrm{Y}^{tw}\big(\mathfrak{osp}_{M|N}\big)$. 
\end{itemize}
\end{remark}

\section{Quantum Sylvester theorem}\label{se:quantumsylvester}
In this section, we introduce the extended twisted super Yangian $\mathcal{X}^{tw}(\mathfrak{osp}_{M|N})$ as stated in \cite[Section 2.13]{Mo07}. Moreover, we derive a Sylvester theorem for the quantum Berezinian on $\mathcal{X}^{tw}(\mathfrak{osp}_{M|N})$. Throughout this section, $\mathcal{Y}^{tw}\big(\mathfrak{osp}_{M|N}\big)$ is consistently referred to the twisted super Yangian. Fix the corresponding supersymmetric block matrix $\mathcal{G}=\mathcal{G}_0$.

\subsection{Extended twisted super Yangian}\label{se:quantumsylvester:1}
\begin{definition}
    The \textit{extended twisted super Yangian} $\mathcal{X}^{tw}\left(\mathfrak{osp}_{M|N}\right)$ is an associative superalgebra generated by coefficients of the formal power series
    \begin{gather*}
        \mathfrak{s}_{ij}(u)=\delta_{ij}+\sum_{r>0}\mathfrak{s}_{ij}^{(r)}u^{-r},\qquad i,j\in I,
    \end{gather*}
     satisfying the only quaternary relation \eqref{Yeqv1}, where $\mathcal{S}(u)=\big(\mathfrak{s}_{ij}(u)\big)_{i,j\in I}$.
\end{definition}

\begin{lemma}
    There exists a formal power series
    \begin{gather*}
        \mathcal{E}(u)\in 1+u^{-1}\mathcal{X}^{tw}(\mathfrak{osp}_{M|N})\left[\left[u^{-1}\right]\right]
    \end{gather*}
    with coefficients in the center of $\mathcal{X}^{tw}\left(\mathfrak{osp}_{M|N}\right)$ such that
    \begin{gather}\label{RQeq}
        \mathcal{E}(u)P^{\iota}:=\widetilde{\mathcal{S}}_2(-u)R(2u)\mathcal{S}_1(u)P^{\iota}=P^{\iota}\mathcal{S}_1(u)R(2u)\widetilde{\mathcal{S}}_2(-u),
    \end{gather}
    where $\widetilde{\mathcal{S}}(u)=\mathcal{S}(u)^{-1}$.
\end{lemma}

\begin{proof}
    Multiply the quaternary relation \eqref{Yeqv1} by $\mathcal{S}_2(v)^{-1}$ on both sides and by $u+v$ on the left to obtain
    \begin{gather*}
        \mathcal{S}_2(v)^{-1}R(u-v)\mathcal{S}_1(u)\left(u+v+P^{\iota}\right)=\left(u+v+P^{\iota}\right)\mathcal{S}_1(u)R(u-v)\mathcal{S}_2(v)^{-1}.
    \end{gather*}
    Taking $v=-u$ in the equation above, we get the second relation of \eqref{RQeq}. Since the action of $P^{\iota}$ on $V\otimes V$ is one-dimensional, there exists an element, denoted by $\mathcal{E}(u)$, that satisfies the equation \eqref{RQeq}.

    By utilizing the argument stated in \cite[Theorem 2.13.3]{Mo07}, one can demonstrate that all the coefficients of $\mathcal{E}(u)$ belong in the center of $\mathcal{X}^{tw}\left(\mathfrak{osp}_{M|N}\right)$.
    
\end{proof}

\begin{theorem}
    The twisted super Yangian $\mathcal{Y}^{tw}\left(\mathfrak{osp}_{M|N}\right)$ is isomorphic to the quotient of $\mathcal{X}^{tw}\left(\mathfrak{osp}_{M|N}\right)$ by the ideal generated by the coefficients of $\mathcal{E}(u)-1+\frac{1}{2u}$.
\end{theorem}

\begin{proof}
    Let $\tilde{\mathfrak{s}}_{ij}(u)$ be the $(i,j)$-th entry of the inverse matrix $\widetilde{\mathcal{S}}(u)=\mathcal{S}(u)^{-1}$.
    It suffices to show that the equation
    \begin{gather*}
        \mathcal{E}(u)=1-\frac{1}{2u}
    \end{gather*}
    is equivalent to the supersymmetric relation \eqref{Yeqv2}. Indeed, apply the operators $P^{\iota}\mathcal{S}_1(u)R(2u)\widetilde{\mathcal{S}}_2(-u)$ and $\mathcal{E}(u)P^{\iota}$ to the basis vector $e_k\otimes e_j$ for $k,j\in I$, we have the equation
    \begin{gather*}
        \delta_{k'j}\theta_{k'}\mathcal{E}(u)w=\sum_{i\in I}(-1)^{|i||j|+|i|+|j|}\left((-1)^{|k|}\theta_i\mathfrak{s}_{i'k}(u)\tilde{\mathfrak{s}}_{ij}(-u)-\frac{1}{2u}(-1)^{|i||k|}\theta_{k'}\mathfrak{s}_{k'i}(u)\tilde{\mathfrak{s}}_{ij}(-u)\right)w,
    \end{gather*}
    for $w=\sum_{r\in I}\theta_r e_{r'}\otimes e_r$. Multiply the equation above by $(-1)^{|k||j|+|j|}\theta_{k'}$ and replace $k$ with $k'$ to get
    \begin{gather*}
        \delta_{kj}\mathcal{E}(u)w=\sum_{i\in I}(-1)^{|i||j|+|k||j|+|i|}\left((-1)^{|k|}\theta_i\theta_{k}\mathfrak{s}_{i'k'}(u)\tilde{\mathfrak{s}}_{ij}(-u)-\frac{1}{2u}(-1)^{|i||k|}\mathfrak{s}_{ki}(u)\tilde{\mathfrak{s}}_{ij}(-u)\right)w.
    \end{gather*}
    This implies that
    \begin{gather*}
        \mathcal{E}(u)=\left(\mathcal{S}^{\iota}(u)-\frac{1}{2u}\mathcal{S}(u)\right)\mathcal{S}(-u)^{-1}.
    \end{gather*}
    Obvious, $ \mathcal{E}(u)=1-\frac{1}{2u}$ is equivalent to \eqref{Yeqv2}.

\end{proof}

\begin{remark}
    We can also define an extended twisted super Yangian $\mathrm{X}^{tw}_{\mathcal{G}}\left(\mathfrak{osp}_{M|N}\right)$ such that the twisted super Yangian $\mathrm{Y}^{tw}_{\mathcal{G}}\left(\mathfrak{osp}_{M|N}\right)$ can be regarded as a quotient of $\mathrm{X}^{tw}_{\mathcal{G}}\left(\mathfrak{osp}_{M|N}\right)$ by the ideal generated by a series of central elements. This statement holds for any supersymmetric block matrix $\mathcal{G}$.
\end{remark}

According to Corollary \ref{S-inver}, we have the following
\begin{lemma}\label{varpi:map}
    There exists an automorphism of superalgebras $\varpi_{M|N}$ over the extended twisted super Yangian $\mathcal{X}^{tw}\left(\mathfrak{osp}_{M|N}\right)$ such that
    \begin{gather*}
        \mathcal{S}(u)\mapsto \widetilde{\mathcal{S}}\left(-u-\frac{M-N}{2}\right).
    \end{gather*}
    In particular, $\varpi_{M|N}^2=1$.
\end{lemma}

Since the definition of quantum Berezinian does not depend on the supersymmetric relation \eqref{Yeqv2}, we can define the quantum Berezinian $\mathfrak{B}^{tw}_{M|N}(u)$ in $u^{-1}$ for the extended twisted super Yangian $\mathcal{X}^{tw}\big(\mathfrak{osp}_{M|N}\big)$ in the same manner as outlined in Section \ref{se:quantumbereianian:1}.

For $k_1,\ldots,k_p,\ell_1,\ldots,\ell_p\in I_+$ and $k_{p+1},\ldots,k_{p+q},\ell_{p+1},\ldots,\ell_{p+q}\in I_-$, introduce the quantum minors $\mathcal{S}_{\ell_1,\ldots,\ell_{p+q}}^{k_1,\ldots,k_{p+q}}(u)$ of matrix $\mathcal{S}(u)$ for $\mathcal{X}^{tw}\left(\mathfrak{osp}_{M|N}\right)$:\footnote{This definition for $p=0$ coincides with the same notation introduced in Section 5.2.}
\begin{gather*}
    \Im A^{(p|q)}\langle \mathcal{S}_{1},\ldots,\mathcal{S}_{p} \rangle\cdot\langle \widetilde{\mathcal{S}}_{p+1},\ldots,\widetilde{\mathcal{S}}_{p+q} \rangle \Im
        \left(e_{\ell_1}\otimes\cdots\otimes  e_{\ell_{p+q}}\right)
       =\sum  \mathcal{S}_{\ell_1,\ldots,\ell_{p+q}}^{k_1,\ldots,k_{p+q}}(u)\left(e_{k_1}\otimes\cdots\otimes e_{k_{p+q}}\right),
\end{gather*}
where $A^{(p|q)}=G^{(p)}\otimes H^{(q)}\otimes 1$. It is easy to see that $\mathfrak{B}^{tw}_{M|N}(u)=\mathcal{S}_{1,\ldots,M+N}^{1,\ldots,M+N}(u)$.

\begin{lemma}\label{KMeqs}
   The following equations hold in $\mathcal{X}^{tw}\left(\mathfrak{osp}_{M|N}\right)\left[\left[u^{-1}\right]\right]$,
    \begin{gather}\label{KleqM}
        \mathfrak{B}^{tw}_{M|N}(u) \varpi_{M|N}\left(\mathcal{S}_{K+1,\ldots,M+N}^{K+1,\ldots,M+N}(-u-\frac{3M-5N}{2}+1)\right)=\mathcal{S}_{1,\ldots,K}^{1,\ldots,K}(u),
    \end{gather}
    for $0\leqslant K\leqslant M$, and
    \begin{gather}\label{KgeqM}
        \mathfrak{B}^{tw}_{M|N}(u) \varpi_{M|N}\left(\mathcal{S}_{K+1,\ldots,M+N}^{K+1,\ldots,M+N}(-u+\frac{M+N}{2}-1)\right)=\mathcal{S}_{1,\ldots,K}^{1,\ldots,K}(u),
    \end{gather}
    for $K>M$.
\end{lemma}
\begin{proof}
    Regard as an operator on $\left(\mathbb{C}^{M|0}\right)^{\otimes M}\otimes \left(\mathbb{C}^{0|N}\right)^{\otimes N}$,
    \begin{gather*}
\Im A^{(M|N)}\mathfrak{B}^{tw}_{M|N}(u)=\Im A^{(M|N)} \langle \mathcal{S}_{1},\ldots,\mathcal{S}_{M} \rangle \cdot\langle \widetilde{\mathcal{S}}_{M+1},\ldots,\widetilde{\mathcal{S}}_{M+N} \rangle.
\end{gather*}

If $K>M$, we have the expression
\begin{equation}\label{AB:eq1}
\begin{split}
    &\Im A^{(M|N)} \langle \mathcal{S}_{1},\ldots,\mathcal{S}_{M} \rangle \cdot\langle \widetilde{\mathcal{S}}_{M+1},\ldots,\widetilde{\mathcal{S}}_K \rangle \cdot \sum_{i=M+1}^{K}(R_{i,K+1}^{\iota}\cdots R_{i,M+N}^{\iota})\\
    =&\Im A^{(M|N)}\mathfrak{B}^{tw}_{M|N}(u)
    \mathcal{S}_{M+N}\left(R_{M+N-1,M+N}^{\iota}\right)^{-1}\mathcal{S}_{M+N-1}\cdots\\
    &\qquad~~\times\mathcal{S}_{K+2}\left(R_{K+1,M+N}^{\iota}\right)^{-1}\cdots \left(R_{K+1,K+2}^{\iota}\right)^{-1}\mathcal{S}_{K+1}
    \end{split}
\end{equation}
Since $\left(R^{\iota}(u)\right)^{-1}=R^{\iota}(-u+M-N)$, the right hand side can be rewritten as
\begin{align*}
    &\Im A^{(M|N)}\mathfrak{B}^{tw}_{M|N}(u) \mathcal{S}_{M+N}^{\bullet}R_{M+N-1,M+N}^{\bullet}\mathcal{S}_{M+N-1}^{\bullet}\cdots \mathcal{S}_{K+2}^{\bullet}R_{K+1,M+N}^{\bullet}\cdots R_{K+1,K+2}^{\bullet}\mathcal{S}_{K+1}^{\bullet} \\
    =&\Im A^{(M|N)}\mathfrak{B}^{tw}_{M|N}(u)\langle \mathcal{S}_{M+N}^{\bullet},\ldots,\mathcal{S}_{K+1}^{\bullet}\rangle,
\end{align*}
where $$\mathcal{S}_i^{\bullet}=\varpi_{M|N}\left(\widetilde{\mathcal{S}}(u_i)\right),\qquad R_{ij}^{\bullet}=R_{ij}^{\iota}(u_i+u_j+M-N), $$
with $u_i$ given as in Section \ref{se:quantumbereianian:1}. Apply both sides of \eqref{AB:eq1} to the vector $v=e_1\otimes\cdots \otimes  e_{M+N}$ and compare the coefficients of $\Im A^{(M|N)}v$, and then we obtain \eqref{KgeqM}.

If $0\leqslant K\leqslant M$, we also have
\begin{equation}\label{AB:eq2}
    \begin{split}
        &\Im A^{(M|N)} \langle \mathcal{S}_{1},\ldots,\mathcal{S}_{K} \rangle  \cdot \sum_{i=1}^{K}(R_{i,K+1}^{\iota}\cdots R_{i,M}^{\iota})\\
    =&\Im A^{(M|N)}\mathfrak{B}^{tw}_{M|N}(u)
    \mathcal{S}_{M}^{\circ}R_{M-1,M}^{\circ}\mathcal{S}_{M-1}^{\circ}\cdots \mathcal{S}_{K+2}^{\circ}R_{K+1,M}^{\circ}\cdots R_{K+1,K+2}^{\circ}\mathcal{S}_{K+1}^{\circ} \\
    &\qquad~~\times \langle \widetilde{\mathcal{S}}_{M+1}^{\bullet},\ldots,\widetilde{\mathcal{S}}_K^{\bullet} \rangle\\
    =&\Im A^{(M|N)}\mathfrak{B}^{tw}_{M|N}(u) \langle \mathcal{S}_{M}^{\circ},\ldots,\mathcal{S}_{K+1}^{\circ} \rangle  \langle\widetilde{\mathcal{S}}_{M+1}^{\bullet},\ldots,\widetilde{\mathcal{S}}_{M+N}^{\bullet} \rangle,
    \end{split}
\end{equation}
where
$$\mathcal{S}_i^{\circ}=\varpi_{M|N}\left(\mathcal{S}(u_i^{\circ })\right),\qquad R_{ij}^{\circ}=R_{ij}^{\iota}(-u_i^{\circ }-u_j^{\circ }),\qquad u_i^{\circ }=-u_i-\frac{M-N}{2}. $$
The equality \eqref{KleqM} is obtained from the action of both sides of \eqref{AB:eq2} to the vector $v$.

\end{proof}

\subsection{Quasi-determinant and quantum Sylvester theorem}
Recall the quasi-determinant of a nonsingular square matrix from \cite{GGRW05}, which allows us to introduce a family of superalgebraic homomorphisms over the extended twisted super Yangians. For a square matrix $X$, denote $X^{ij}$ as the minor submatrix obtained from $X$ by deleting the $i$-th row and the $j$-th column.
\begin{definition}\label{quasi-det}
    Let $X=(x_{ij})$ be a nonsingular ($m\times m$)-matrix over a ring with unity such that its $(i,j)$-th minor submatrix is invertible. Then the $(i,j)$-th \textit{quasi-determinant} of $X$ is defined by the following formula
    \begin{equation*}
        |X|^{ij}=\left|
  \begin{array}{ccccc}
    x_{11} & \cdots & x_{1j} & \cdots & x_{1m} \\
    \vdots &  & \vdots & & \vdots \\
    x_{i1} & \cdots & \boxed{x_{ij}} & \cdots & x_{im} \\
    \vdots &  & \vdots & & \vdots \\
    x_{m1} & \cdots & x_{mj} & \cdots & x_{mm}  \\
  \end{array}
\right|:=x_{ij}-r_i^j(X^{ij})^{-1}c_j^i,
    \end{equation*}
    where $r_i^j$ is the row matrix obtained from the $i$-th row of $X$ by deleting the element $x_{ij}$ and $c_j^i$ is the column matrix obtained from the $j$-th column of $X$ by deleting the element $x_{ij}$.
\end{definition}

By the defining relation \eqref{Yeqv1}, for any $K\in\mathbb{N}$ there is a natural superalgebra homomorphism $\nu_K$ such that
\begin{gather*}
    \nu_K:\ \mathcal{X}^{tw}(\mathfrak{osp}_{M|N})\rightarrow \mathcal{X}^{tw}(\mathfrak{osp}_{K+M|N}),\quad \mathfrak{s}_{ij}(u)\mapsto \mathfrak{s}_{K+i,K+j}(u),\quad 1\leqslant i,j\leqslant M+N.
\end{gather*}
Introduce the composite
\begin{gather*}
    \psi_K=\varpi_{K+M|N}\,\circ \,\nu_K\,\circ\,\varpi_{M|N},
\end{gather*}
where $\varpi_{M|N}$ is defined in Lemma \ref{varpi:map}. The following proposition gives the image of the generating series for $\mathcal{X}^{tw}\big(\mathfrak{osp}_{M|N}\big)$ in terms of quasi-determinants. Since the constant term of each series $\mathfrak{s}_{ij}(u)$ is $\delta_{ij}$, these quasi-determinants are well-defined.

\begin{proposition}
For any $1\leqslant i,j\leqslant M+N$ and $K\geqslant 0$, we have
    \begin{equation}\label{psi:K}
\psi_K:\ \mathfrak{s}_{ij}(u+\frac{K}{2})\mapsto
\left|
  \begin{array}{cccc}
    \mathfrak{s}_{11}(u) & \ldots & \mathfrak{s}_{1K}(u) & \mathfrak{s}_{1,K+j}(u) \\
    \ldots & \ldots & \ldots & \ldots \\
    \mathfrak{s}_{K1}(u) & \ldots & \mathfrak{s}_{KK}(u) & \mathfrak{s}_{K,K+j}(u) \\
    \mathfrak{s}_{K+i,1}(u) & \ldots & \mathfrak{s}_{K+i,K}(u) & \boxed{\mathfrak{s}_{K+i,K+j}(u)}\\
  \end{array}
\right|.
\end{equation}
\end{proposition}
\begin{proof}
    According to the statements in \cite[Lemma 1.11.1]{Mo07}, for any nonsingular block matrix
    \begin{equation*}
        \begin{pmatrix}
            A &B \\
            C &D
        \end{pmatrix}
    \end{equation*}
    over a ring with unity, such that the matrices
    $A$ and $D$ are invertible, the matrices
    \begin{gather*}
        A-BD^{-1}C,\qquad D-CA^{-1}B
    \end{gather*}
    are also invertible. Moreover,
    \begin{equation*}
        \begin{pmatrix}
            A &B \\
            C &D
        \end{pmatrix}^{-1}=\begin{pmatrix}
            (A-BD^{-1}C)^{-1} & -A^{-1}B(D-CA^{-1}B)^{-1} \\
            -D^{-1}C(A-BD^{-1}C)^{-1}  & (D-CA^{-1}B)^{-1}
        \end{pmatrix}.
    \end{equation*}

   Let
    \begin{equation*}
        \mathcal{S}(u)=\begin{pmatrix}
            A(u)  &  B(u) \\
            C(u)  &  D(u)
        \end{pmatrix}\in \mathrm{End}(\mathbb{C}^{K+M|N})\otimes\mathcal{X}^{tw}(\mathfrak{osp}_{K+M|N}),
    \end{equation*}
    where
    \begin{equation*}
        \begin{aligned}
            &A(u)=\left(\mathfrak{s}_{ij}(u)\right)_{1\leqslant i,j\leqslant K}, \\
            &C(u)=\left(\mathfrak{s}_{ij}(u)\right)_{K+1\leqslant i\leqslant K+M+N,\ 1\leqslant j\leqslant K},
        \end{aligned}
        \qquad
        \begin{aligned}
            &B(u)=\left(\mathfrak{s}_{ij}(u)\right)_{1\leqslant i\leqslant K,\ K+1\leqslant j\leqslant K+M+N}, \\
            &D(u)=\left(\mathfrak{s}_{ij}(u)\right)_{K+1\leqslant i,j\leqslant K+M+N}.
        \end{aligned}
    \end{equation*}
    By its definition, the homomorphism $\psi_K$ maps $\varpi_{M|N}(\mathcal{S}(u))=\widetilde{\mathcal{S}}\left(u'+\frac{K}{2}\right)$
    to
    $$\left(D(u')-C(u')A^{-1}(u')B(u')\right)^{-1}\quad\text{with}\quad u'= -u-\frac{M+K-N}{2}.$$
    Then we have
    \begin{gather*}
        \psi_K:\ \mathcal{S}\left(u+\frac{K}{2}\right)\mapsto D(u)-C(u)A^{-1}(u)B(u).
    \end{gather*}
    Taking the $(i,j)$-th entry and using Definition \ref{quasi-det}, we immediately get \eqref{psi:K}.

\end{proof}

For any even integers $L\geqslant 0$, there is a natural homomorphism $\eta_{L}$ from $\mathcal{X}^{tw}\left(\mathfrak{osp}_{M|N}\right)$ to $\mathcal{X}^{tw}\left(\mathfrak{osp}_{M|L+N}\right)$ defined by
$$\mathfrak{s}_{ij}(u)\mapsto \mathfrak{s}_{ij}(u),\qquad 1\leqslant i,j\leqslant M+N.$$
Note that the homomorphism $\psi_K$ maps $\tilde{\mathfrak{s}}_{ij}(u+\frac{K}{2})$ to $\tilde{\mathfrak{s}}_{K+i,K+j}(u)$, hence, the image of $\psi_K$ in $\mathcal{X}^{tw}\left(\mathfrak{osp}_{K+M|L+N}\right)$ commutes with the sub-superalgebra generated by all $\eta_L\big(\mathfrak{s}_{ab}(v)\big)$ for each $L$ where $\mathfrak{s}_{ab}(v)\in \mathcal{X}^{tw}\left(\mathfrak{osp}_{K+M|N}\right)$ and $1\leqslant a,b\leqslant K$. Specially, the image
$\psi_K\left(\mathcal{X}^{tw}\left(\mathfrak{osp}_{M|N}\right)\right)$
commutes with the sub-superalgebra of $\mathcal{X}^{tw}\left(\mathfrak{osp}_{K+M|N}\right)$ generated by the set $\big\{\,\mathfrak{s}_{ab}^{(r)}\,\big|\,1\leqslant a,b\leqslant K,\ r>0\,\big\}$.

\begin{corollary}
   For any integers $K_1,K_2\geqslant 0$, we have
    \begin{gather*}
        \psi_{K_1}\circ \psi_{K_2}=\psi_{K_1+K_2}=\psi_{K_2}\circ \psi_{K_1}.
    \end{gather*}
\end{corollary}

Now we have the following Sylvester theorem for quantum Berezinian over the extended twisted super Yangians.
\begin{theorem}
   It holds for $K\geqslant 0$,
    \begin{gather*}
        \psi_K\left(\mathfrak{B}^{tw}_{M|N}(u)\right)=\mathfrak{B}^{tw}_{K+M|N}\left(u+\frac{3K}{2}\right)\cdot \left(\mathcal{S}_{1\cdots K}^{1\cdots K}\left(u+\frac{3K}{2}\right)\right)^{-1}.
    \end{gather*}
\end{theorem}
\begin{proof}
    This is an immediately conclusion of Lemma \ref{KMeqs} and the definition of $\psi_K$.
\end{proof}

\vspace{2em}
The authors have no conflicts of interest to declare that are relevant to this article.

\end{document}